\documentclass[smallextended,numbook,runningheads]{svjour3}     
\smartqed  
\usepackage{graphicx}
\usepackage{amsmath}
\usepackage{epstopdf} 
\usepackage{mathptmx}      
\newcommand{\R}{{\mathbf{R}}}

\newcommand{\E}{{\mathbf{E}}}
\newcommand{\N}{{\mathbf{N}}}

\newcommand{\F}{{\mathcal{F}}} 
\renewcommand{\P}{{\mathbf{P}}} 

\newcommand{\Tr}{{\textrm{Tr}}}
\newcommand{\LOR}{L^2(\Omega;\R^m)}

\newcommand{\diff}[1]{\,\mathrm{d}#1}

\newcommand{\triple}{{\vert\kern-0.25ex\vert\kern-0.25ex\vert}}
\newcommand{\HS}{{\mathrm{HS}}}

\spnewtheorem{assumption}{Assumption}[section]{\bf}{\rm}
\spnewtheorem{prop}{Proposition}[section]{\bf}{\it}

\begin{document}

\title{Mean-square convergence of the BDF2-Maruyama and backward Euler schemes
for SDE satisfying a global monotonicity
condition}

\titlerunning{Mean-square convergence of the BDF2-Maruyama and backward Euler
schemes} 

\author{Adam Andersson         \and
        Raphael Kruse 
}


\institute{Adam Andersson \at
              Technische Universit\"at Berlin\\
              Institut f\"ur Mathematik, Secr. MA 5-3\\
              Stra\ss e des 17.~Juni 136\\
              DE-10623 Berlin\\
              Germany \\
              Tel.: +49 (0) 30 314 - 2 96 80\\
              Fax: +49 (0) 30 314 - 2 89 67\\
              \email{andersson@math.tu-berlin.de}           
           \and
             Raphael Kruse \at
             Technische Universit\"at Berlin\\
             Institut f\"ur Mathematik, Secr. MA 5-3\\
             Stra\ss e des 17.~Juni 136\\
             DE-10623 Berlin\\
             Germany\\
             Tel.: +49 (0) 30 314 - 2 33 54\\
             Fax: +49 (0) 30 314 - 2 89 67\\
             \email{kruse@math.tu-berlin.de}
}

\date{Received: date / Accepted: date}

\maketitle

\begin{abstract}
  In this paper the numerical approximation of stochastic differential
  equations satisfying a global monotonicity condition is studied. The
  strong rate of convergence with respect to the mean square norm is determined
  to be $\frac{1}{2}$ for the two-step BDF-Maruyama scheme and for the backward
  Euler-Maruyama method. In particular, this is the first paper which proves a
  strong convergence rate for a multi-step method applied to equations with
  possibly superlinearly growing drift and diffusion coefficient functions. 
  We also present numerical experiments for the $\tfrac32$-volatility model
  from  finance and a two dimensional problem related to Galerkin approximation of SPDE, which verify our results in practice and indicate that the
  BDF2-Maruyama method offers advantages over Euler-type methods if the
  stochastic differential equation is stiff or driven by a noise with small
  intensity.

\end{abstract}

\section{Introduction}
\label{sec:intro}

Strong convergence rates of numerical approximations to stochastic differential
equations (SDEs) are a well studied topic. Under a global Lipschitz condition on
the coefficients the picture is rather complete, both for one-step methods
\cite{kloeden1999}, \cite{milstein2004} and multi-step methods
\cite{buckwar2006}, \cite{kruse2011}. Many important equations in application
have coefficients that do 
not satisfy the global Lipschitz condition, and it is therefore important to
study a more general setting. Many convergence results for explicit and
implicit one-step methods have also been proven for equations without the
global Lipschitz condition, see for instance 
\cite{beynisaakkruse2014}, \cite{higham2002b}, \cite{hu1996},
\cite{hutzenthaler2014c}, \cite{hutzenthaler2014a}, \cite{hutzenthaler2012},
\cite{mao2013a}, \cite{sabanis2013b}, \cite{tretyakov2013}.  
In the present paper we determine the strong rate $\tfrac12$ for the backward
Euler-Maruyama method (BEM), in the mean-square norm, which improves
\cite{mao2013a} in terms of a weaker assumption on the coefficients.

For multi-step schemes, on the other hand, there are no previously known
results on strong convergence for equations with coefficients not satisfying a
global Lipschitz condition. In this paper we determine the strong rate
$\tfrac12$ for the BDF2-Maruyama scheme for equations whose, possibly
superlinearly growing, coefficient functions satisfy a global monotonicity
condition.  Backward difference formulas (BDF) are popular in applied sciences
for the approximation of stiff equations, see \cite{buckwar2006}
for a list of references to such works.

Let $d,m \in \N$, $T > 0$ and $(\Omega,\F,(\F_t)_{t \in [0,T]}, \P)$ be a
filtered probability space satisfying the usual conditions, on which an
$\R^d$-valued standard $(\F_t)_{t \in [0,T]}$-Wiener process $W \colon [0,T]
\times \Omega \to \R^d$ is defined. We consider the
equation 
\begin{align}
  \label{eq:SDE}
  X(t) = X_0 + \int_0^t f(X(s)) \diff{s} + \int_0^t g(X(s))
  \diff{W(s)},\quad 
  t\in[0,T],
\end{align}
with drift $f \colon \R^m \to \R^m$ and diffusion coefficient function $g
\colon \R^m \to \R^{m\times d}$. The functions $f$ and $g$ are assumed to
satisfy a global monotonicity, a coercivity and a local Lipschitz condition in
Assumption~\ref{as:fg} below.  
The initial condition fulfills $X_0\in L^2(\Omega,\mathcal{F}_0,\P;\R^m)$ with
some additional integrability, admitting higher moments of the solution.

For a given equidistant time step size $h \in (0,1)$ we discretize the exact
solution to \eqref{eq:SDE} along the temporal grid $\tau_h = \{ t_n = nh \; :
\; n = 0,1,\ldots,N_h \}$. Here $N_h \in \N$ is uniquely determined by the
inequality $t_{N_h} \le T < t_{N_{h} + 1}$. We set $\Delta_h W^j := W(t_j) -
W(t_{j-1})$ for $j\in\{1,\dots,N_h\}$. We consider discretizations by means of
the \emph{ backward Euler-Maruyama method}
\begin{equation}
  \label{eq:BEM}
  X_h^{j} - X_h^{j-1} 
  = h f(X_h^j) + g( X_h^{j-1} ) \Delta_h W^j,
  \quad j\in\{1,\dots,N_h\},
\end{equation}
with $(X_h^0)_{h\in(0,1)}$ satisfying $\E[\|X_h^0-X_0\|^2]=\mathcal{O}(h)$, and
by means of the \emph{BDF2-Maru\-yama scheme} from \cite{buckwar2006}. The latter
is given by the recursion 
\begin{equation}
  \label{eq:BDF2}
  \frac{3}{2} X_h^{j} - 2 X_h^{j-1} + \frac{1}{2} X_h^{j-2} 
  = h f(X_h^j) + \frac{3}{2} g( X_h^{j-1} ) \Delta_h W^j
  - \frac{1}{2}  g(X_h^{j-2} ) \Delta_h W^{j-1},
\end{equation}
for $j\in\{2,\dots,N_h\}$, with initial values $(X_h^0,X_h^1)$ where
$(X_h^0)_{h\in(0,1)}$ is the same as above and $(X_h^1)_{h\in(0,1)}$ is
determined, for instance, by one step of the backward Euler scheme or some
other one-step method satisfying  $\E[\|X_h^1-X(h)\|^2]=\mathcal{O}(h)$.
In practice, the implementation of the methods \eqref{eq:BEM} and
\eqref{eq:BDF2} often requires to solve a nonlinear equation in each time step.
In Section \ref{sec:wellposedness} we discuss that under our assumptions a
solution does indeed always exists provided the step size $h$ is small enough.
The choice of the root-finding algorithm may depend on the coefficient function
$f$ and its smoothness. We refer to \cite{ortega2000} for  
a collection of such methods.

We prove that for $(X_h^j)_{j\in \{0,\dots,N_h\}, h\in(0,1)}$, determined
either by \eqref{eq:BEM} or \eqref{eq:BDF2}, and $X$ being the solution to
\eqref{eq:SDE}, there exist a constant $C$ such that the following 
mean-square convergence holds: 
\begin{align}
  \label{eq:msconv}
  \max_{j\in\{0,\dots,N_h\}}
  \big\|
    X(t_j)-X_h^j
  \big\|_{L^2(\Omega;\R^m)}
  \leq
  C
  \sqrt{h}
  ,\quad
  h\in(0,1).
\end{align}
The precise statements of our convergence results are found in
Theorems~\ref{thm:BEMconvergence} and \ref{thm:convergence}.
The proofs are based on two elementary identities: for all $u_1,u_2\in\R^m$ it
holds that 
\begin{equation}
  \label{eq:emmrich1}
  \begin{split}
    2(
      u_2 - u_1 , u_2
    )=
      |u_2|^2-|u_1|^2
      +
      |u_2-u_1|^2,
  \end{split}
\end{equation}
and for all $u_1,u_2,u_3\in\R^m$ it
holds that 
\begin{equation}
  \label{eq:emmrich2}
  \begin{split}
    4\Big(
    &  \frac32u_3
    -
    2 u_2
    +
    \frac12 u_1
    ,
    u_3
    \Big)\\
    & =
      |u_3|^2-|u_2|^2
      +
      |2u_3-u_2|^2 - |2u_2-u_1|^2
      +
      |u_3-2u_2+u_1|^2,
  \end{split}
\end{equation}
found in \cite{emmrich2009}, which has been derived from results
on $G$-stability for linear multi-step methods, see 
\cite{girault1979,stuart1996}. Up to the best of our knowledge
\eqref{eq:emmrich2} has not previously been used in the study of the BDF2
scheme for stochastic differential equations. 

The paper is organized as follows: Section~\ref{sec:setting} contains notation
and our precise assumptions on the coefficients $f$ and $g$ in \eqref{eq:SDE}.
We cite well known results on existence, uniqueness and moment bounds for the
solution under these conditions. A well-posedness result for general implicit 
stochastic difference equations is proved in
Section~\ref{sec:wellposedness}. Sections~\ref{sec:BEM} and~\ref{sec:BDF2}
contain the analysis of the backward Euler-Maruyama and the BDF2-Maruyama
schemes, respectively. Subsections \ref{subsec:BEMapriori} and
\ref{subsec:existence} contain a priori estimates for the respective schemes,
in Sections~\ref{subsec:BEMstability} and \ref{subsec:stability} stability
results are proved, while Subsections~\ref{subsec:BEMcons} and
\ref{subsec:consistency} are concerned with the consistency of the two schemes.
The two main results on the strong mean-square convergence rate are stated in
Sections~\ref{subsec:BEMconvergence} and \ref{subsec:convergence},
respectively. Further, in Subsection~\ref{subsec:BDF2initial} we have a closer
look on the second initial value for
the BDF2-Maruyama scheme and it is shown that using one step of the BEM method
is a feasible choice. Section~\ref{sec:numerics} contains numerical experiments
involving the $\frac{3}{2}$-volatility model from finance which verify our
theoretical results and indicate that the BDF2-Maruyama method performs better
than Euler-type methods in case of stiff problems or equations with a small
noise intensity.

\section{Setting and preliminaries}
\label{sec:setting}

\subsection{Notation and function spaces}
\label{subsec:notation}
Let $(\cdot,\cdot)$ and $|\cdot|$ denote the scalar product and
norm in $\R^m$ and let $|\cdot|_{\HS}$ denote the Hilbert-Schmidt norm on
the space $\R^{m\times d}$ of all $m$ times $d$ matrices, i.e.,
$|S|_{\HS}=\sqrt{\Tr(S^*S))}$ for $S\in\R^{m\times d}$. 

Let $(\Omega,\F,(\F_t)_{t \in [0,T]}, \P)$ be a filtered probability
space satisfying the usual conditions.
For $p \in [1,\infty)$ and a sub-$\sigma$-field $\mathcal{G} \subset
\F$ we denote by $L^p(\Omega, \mathcal{G}, \P; E)$ the Banach space of all
$p$-fold integrable, $\mathcal{G} / \mathcal{B}(E)$-measurable random variables
taking values in a Banach space $(E, | \cdot |_E)$ with norm 
\begin{align*}
  \| Z \|_{L^p(\Omega,\mathcal{G},\P;E)} = \big( \E \big[ | Z |^p_E \big]
  \big)^{\frac{1}{p}}, \quad Z \in L^p(\Omega,\mathcal{G},\P;E).
\end{align*}
If $\mathcal{G} = \F$ we write $L^p(\Omega;E) := L^p(\Omega,
\F, \P; E)$. If $p = 2$ and $E = \R^m$ we obtain the Hilbert space $\LOR$ with
inner product and norm
\begin{align*}
  \big \langle X, Y \big \rangle = \E \big[ ( X, Y ) \big]
  ,\quad
  \|X\|=\sqrt{\langle X,X\rangle},
\end{align*}
for all $X, Y \in \LOR$. We denote by $\triple \cdot \triple$ the norm in
$L^2(\Omega;\R^{m\times d})$, i.e., $\triple Z \triple =
(\E[|Z|_{\HS}^2])^\frac12$ for $Z\in L^2(\Omega;\R^{m\times d})$.  

We next introduce notation related to the numerical discretizations. Recall
from Section~\ref{sec:intro} the temporal grids $\tau_h$, $h\in(0,1)$. For
$h\in(0,1)$ and 
$j\in\{0,\dots,N_h\}$, we denote by 
\begin{align*}
  P_h^j\colon L^2(\Omega,\mathcal{F},\P;\R^m)\to
  L^2(\Omega,\mathcal{F}_{t_j},\P;\R^m), 
\end{align*} 
the orthogonal projector onto the closed sub-space
$L^2(\Omega,\mathcal{F}_{t_j},\P;\R^m)$, which is also known
as the conditional expectation. More precisely, for $Y\in L^2(\Omega;\R^m)$ we
set $P_h^j\,Y=\E[Y|\mathcal{F}_{t_j}]$. We introduce the spaces
$(\mathcal{G}_h^2)_{h\in(0,1)}$ of all adapted grid functions, which enjoy the
following integrability properties
\begin{align*}
  \mathcal{G}^2_h
  :=
  \big\{
    Z
    &\colon
    \{0,\dots,N_h\}\times\Omega\to\R^m
    \, : \, Z^n, f(Z^n) \in
    L^2(\Omega,\mathcal{F}_{t_n},\mathbf{P};\R^m),\\
    & g(Z^n) \in  L^2(\Omega,\mathcal{F}_{t_n},\mathbf{P};\R^{m \times d})
    \textrm{ for }
    n\in\{0,\dots,N_h\}
  \big\}.
\end{align*}
These will play an important role in the error analysis.

\subsection{Setting}
\label{subsec:setting}
Consider the setting introduced in Section~\ref{sec:intro}.
We now formulate our assumptions on the initial condition and the
coefficient functions $f$ and $g$ which we work with throughout this paper.

\begin{assumption}
  \label{as:fg}
  There exists $q \in [1, \infty)$ such that the initial condition
  $X_0\colon\Omega\to\R^m$ satisfies $X_0\in 
  L^{4q-2}(\Omega,\F_{0},\P;\R^m)$. 
  Moreover, the mappings $f \colon \R^m \to \R^m$ and
  $g \colon \R^m \to \R^{m \times d}$, are
  continuous and there exist $L \in (0,\infty)$ and 
  $\eta \in (\tfrac12,\infty)$ such that for all
  $x_1,x_2 \in \R^m$ it holds 
  \begin{align}
    \label{eq:onesided}
    &\big( f(x_1) - f(x_2), x_1-x_2 \big) + \eta  
    \big| g(x_1) - g(x_2) \big|_{\HS}^2 \le L | x_1 - x_2 |^2,\\
    \label{eq:loc_lip_f}
    &\big| f(x_1) - f (x_2) \big| \le L \big( 1 +
    |x_1|^{q-1} + |x_2 |^{q-1} \big)  | x_1 - x_2 |,
  \end{align}
  where $q \in [1,\infty)$ is the same as above. Further, it holds
  for all $x\in\R^m$ that
  \begin{align}
    \label{eq:bound_fg}
    &\big(
      f(x),x
    \big)
    +
    \frac{4q-3}2
    \big|
      g(x)
    \big|_{\HS}^2
    \leq
    L\big( 1 + |x|^2\big). 
  \end{align}
\end{assumption}

Assumption~\ref{as:fg} guarantees the existence of an up
to modification unique adapted solution $X\colon [0,T]\times\Omega\to \R^m$ to
\eqref{eq:SDE} with continuous sample paths, satisfying  
\begin{align}
  \label{eq:moment}
  \sup_{t\in[0,T]}
  \|X(t)\|_{L^{4q-2}(\Omega;\R^m)}
  <\infty,
\end{align}
see, e.g., \cite[Chap.~2]{mao1997}. In the proof of
Theorem~\ref{thm:consistency} on the consistency of the BDF2 scheme, the
$L^{4q-2}(\Omega;\R^m)$-moment bound is of importance in order to apply
the bounds \eqref{eq:bound_f}, \eqref{eq:bound_g} below. 

For later reference we note several consequences of Assumption~\ref{as:fg}.
>From \eqref{eq:loc_lip_f} we deduce the following polynomial growth bound: 
\begin{align}
  \label{eq:poly_growth}
  \big| f(x) \big| \le \tilde{L} \big( 1 + |x|^q \big), \quad x \in \R^m,
\end{align}
where $\tilde{L} = 2L + |f(0)|$. Indeed, \eqref{eq:loc_lip_f} implies that
\begin{align*}
  \big| f(x) \big| &\le \big| f(x) - f(0) \big| + \big|f(0) \big| \le L \big( 1
  + |x|^{q-1} \big) |x| + \big|f(0)\big| \\
  &\le \big( 2L + |f(0)| \big) \big( 1 + |x|^q \big), \quad x \in
  \R^m.
\end{align*}
Moreover, from \eqref{eq:onesided} followed by a use of \eqref{eq:loc_lip_f} it
holds, for $x_1,x_2\in\R^m$, that 
\begin{align*}
  \big|
    g(x_1) - g(x_2)
  \big|_{\HS}^2
&\leq
  \frac{L}\eta
  |x_1-x_2|^2
  +
  \frac1\eta
  \big|
    \big(
      f(x_1)-f(x_2),x_1-x_2
    \big)
  \big|\\
&\leq
  \frac{L}\eta
  \big(
    2 + |x_1|^{q-1}+|x_2|^{q-1}
  \big)
  |x_1-x_2|^2.
\end{align*}
This gives the local Lipschitz bound
\begin{align}
  \label{eq:loc_lip_g}
  \big|
    g(x_1) - g(x_2)
  \big|_{\HS}^2
  \leq
  \frac{2L}\eta
  \big( 1 +
    |x_1|^{q-1} + |x_2 |^{q-1} 
  \big)  
  | x_1 - x_2 |^2
  ,\quad x_1,x_2\in\R^m,
\end{align}
and, in the same way as above, the polynomial growth bound
\begin{align}
  \label{eq:poly_growth_g}
  \big|
    g(x) 
  \big|_{\HS}
  \leq
  \bar{L} \big( 1 + |x|^{\frac{q+1}{2}} \big)
  ,\quad x \in\R^m,
\end{align}
where $\bar{L} = 2 \sqrt{\frac{2 L}{\eta}} + |g(0)|_{\HS}$.
Finally, we note for later use that the \emph{restriction} $X|_{\tau_h}$ of the
exact solution to the time grid $\tau_h$, given by
\begin{align*}
  \big[X|_{\tau_h}\big]^j := X(jh), \quad j \in \{0,\ldots,N_h\},
\end{align*}
is an element of the space $\mathcal{G}_h^2$ for every $h \in (0,1)$. This
follows directly from \eqref{eq:moment} and the growth bounds
\eqref{eq:poly_growth} and \eqref{eq:poly_growth_g}.

\subsection{Preliminaries}
\label{subsec:preliminaries}
Here we list some basic results that we use in this paper. Frequently, we apply
the Young inequality and the weighted Young inequality 
\begin{align}
  \label{eq:Young}
  a b \le \frac{a^2}{2}  + \frac{b^2}{2}  \quad\textrm{and}\quad
  a b \le \frac{\nu}{2} a^2 + \frac{1}{2 \nu} b^2,
\end{align}
which holds true for all $a, b \in \R$ and $\nu >0$. We make use of the
following discrete version of Gronwall's Lemma: If $h>0$,
$a_1,\dots,a_{N_h},b,c \in [0,\infty)$, then 
\begin{align}\label{eq:Gronwall}
  \forall n\in\{1,\dots,N_h\}: \ a_n \leq c + bh\sum_{j=1}^{n-1}a_{j}
  \quad
  \textrm{implies}
  \quad
  \forall n\in\{1,\dots,N_h\}: \ a_n\leq ce^{bt_n}.
\end{align}

Finally we cite a standard result from nonlinear analysis which we use for the
well-posedness of the numerical schemes, see for instance
\cite[Chap.~6.4]{ortega2000} or \cite[Thm.~C.2]{stuart1996}:

\begin{prop}
  \label{prop:homeo}
  Let $G \colon \R^m \to \R^m$ be a continuous mapping satisfying for some $c
  \in (0, \infty)$
  \begin{align*}
    \big( G(x_1) - G(x_2), x_1 - x_2 \big) \ge c | x_1 - x_2 |^2, \quad x_1,
    x_2 \in \R^m.
  \end{align*}
  Then $G$ is a homeomorphism with Lipschitz continuous inverse. 
  In particular, it holds
  \begin{align*}
    \big| G^{-1}(y_1) - G^{-1}(y_2) \big| \le \frac{1}{c} | y_1 - y_2|
  \end{align*}
  for all $y_1, y_2 \in \R^m$.
\end{prop}

\section{A well-posedness result for stochastic difference equations}
\label{sec:wellposedness}

In this section we prove existence and uniqueness of solutions to general
stochastic $k$-step difference equations. This result applies in particular to
all implicit linear multi-step schemes for SDE with coefficients satisfying
Assumption~\ref{as:fg} including the backward Euler-Maruyama method, the
Crank-Nicolson scheme, the $k$-step BDF-schemes, and the $k$-step 
Adams-Moulton methods. We refer the reader to
\cite{buckwar2006,kruse2011} for a thorough treatment of these schemes for
stochastic differential equations with Lipschitz continuous coefficients. 

\begin{theorem}\label{thm:wellposedness}
  Let the mappings $f$ and $g$ satisfy Assumption~\ref{as:fg} with $q \in
  [1,\infty)$ and $L \in (0,\infty)$, let $k\in\N$,
  $\alpha_0,\dots,\alpha_{k-1}, \beta_0,\dots,\beta_{k-1}, \gamma_0,\dots,
  \gamma_{k-1}\in\R$, $\alpha_k=1$, $\beta_k \in (0,\infty)$, and 
  $h_1 \in (0, \tfrac1{\beta_kL})$ with $k h_1<T$. Assume that
  initial values 
  $U_h^\ell \in L^{2}(\Omega,\mathcal{F}_{t_\ell},\P;\R^m)$ are given with
  $f(U_h^\ell) \in L^{2}(\Omega,\mathcal{F}_{t_\ell},\P;\R^m)$ and $g(U_h^\ell)
  \in L^{2}(\Omega,\mathcal{F}_{t_\ell},\P;\R^{m \times d})$ for all
  $\ell\in\{0,\dots,k-1\}$. Then, for every $h\in(0,h_1]$ there exists a unique
  family of adapted random variables $U_h\in\mathcal{G}_h^2$ satisfying  
  \begin{equation}
    \label{eq:RandDiffEq}
    \begin{split}
      \sum_{\ell=0}^{k} \alpha_{k-\ell} U_h^{j-\ell}
      &= h \sum_{\ell=0}^k \beta_{k-\ell} f(U_h^{j-\ell}) +
      \sum_{\ell=1}^k\gamma_{k-\ell} g( U_h^{j-\ell} ) \Delta_h W^{j-\ell+1},
    \end{split}
  \end{equation}  
  for $j\in\{k,\dots,N_h\}$. In particular, it holds true that $U_h^j,\, f(U_h^j)
  \in L^2(\Omega, \F_{t_j}, \P; \R^m)$ and $g(U_h^j) \in L^2(\Omega, \F_{t_j},
  \P; \R^{m \times d})$ for all $j \in \{k,\dots,N_h\}$.
\end{theorem}

\begin{proof}
  Let $F_h \colon \R^m \to \R^m$, $h\in(0,h_1]$, be the mappings defined by 
  \begin{align*}
    F_h(x) := x - h \beta_k f(x), \quad x \in \R^m; \quad h\in(0,h_1].
  \end{align*}
  Note that for every $h \in (0,h_1]$ it holds that $1-\beta_k h L\ge
  1-\beta_k h_1 L > 0$ and 
  from the global monotonicity condition \eqref{eq:onesided} we have that 
  \begin{align*}
    \big( F_h(x_1) - F_h(x_2), x_1 - x_2 \big) 
    &= |x_1 - x_2 |^2 - \beta_k h \big( f(x_1) - f(x_2), x_1 - x_2 \big)\\
    &\ge \big(1 - \beta_k h_1 L \big) |x_1 - x_2 |^2.
  \end{align*}
  Consequently, by Proposition~\ref{prop:homeo} the inverse $F_h^{-1}$ of $F_h$
  exists for every $h \in (0, h_1]$ and is globally Lipschitz continuous
  with Lipschitz constant $(1-h_1 \beta L)^{-1}$. Using these properties
  and the fact that $\alpha_k = 1$ we can rewrite \eqref{eq:RandDiffEq} as 
  \begin{equation}
    \label{eq:RDEinverse}
    \begin{split}
      F_h(U_h^{j}) =  R_h^j \quad \Longleftrightarrow
      \quad U_h^j= F_h^{-1}(R_h^j)
    \end{split}
  \end{equation}
  for all $j\in\{k,\dots,N_h\}$, where 
  \begin{align*}
    R_h^j:=- \sum_{\ell=1}^{k} \alpha_{k-\ell} U_h^{j-\ell} 
    + h \sum_{\ell=1}^{k} \beta_{k-\ell} f(U_h^{j-\ell}) 
    + \sum_{\ell=1}^k\gamma_{k-\ell} g( U_h^{j-\ell} )\Delta_h W^{j-\ell+1},
  \end{align*}
  for $j\in\{k,\dots,N_h\}$. Therefore, by \eqref{eq:RDEinverse} and the
  continuity of $F_h^{-1}$ we have that for every $h\in (0,h_1]$, $U_h$ is an
  adapted collection of random variables, uniquely determined by the initial
  values $U_h^0,\dots,U_h^{k-1}$. 

  In order to prove that $U_h\in\mathcal{G}_h^2$ for all $h\in(0,h_1]$, by means
  of an induction argument, we introduce
  $U_h^{j,n}:=U_h^j\mathbf{1}_{\{0,\dots,n-1\}}(j)$. By the assumptions on the
  initial values $U_h^0,\dots, U_h^{k-1}$ it holds that  
  $(U_h^{j,k})_{j \in \{0,\ldots,N_h\}} \in \mathcal{G}_h^2$. For the induction
  step, we now assume that $(U_h^{j,n})_{j \in \{0,\ldots,N_h\}}
  \in\mathcal{G}_h^2$ for some $n\in\{k,\dots,N_h\}$. This assumption
  and the fact that 
  \begin{align*}
    \big\| g( U_h^{j-\ell} )\Delta_h W^{j-\ell+1} \big\|^2 = h \triple
    g(U_h^{j-\ell} ) \triple^2,
  \end{align*}
  imply immediately that $R_h^n\in L^2(\Omega,\mathcal{F}_{t_n},\P;\R^m)$.
  Thus, from the linear growth of $F_h^{-1}$ we get $U_h^{n}=F_h^{-1}(R_h^n)\in
  L^2(\Omega,\mathcal{F}_{t_n},\P;\R^m)$.

  In addition, we recall from \cite[Cor.~4.2]{beynisaakkruse2014} the fact that
  under Assumption~\ref{as:fg} the mapping $h^{\frac{1}{2}} g \circ F_h^{-1}$ is
  also globally Lipschitz continuous with a Lipschitz constant independent of
  $h$. More precisely, there exists a constant $C$ such that for all $h \in (0,
  h_1]$ and all $x_1,x_2 \in \R^m$ it holds true that 
  \begin{align}
    \label{eq:gFLip}
    h \big| g( F_h^{-1}(x_1) ) - g( F_h^{-1}(x_2) ) \big|_{\HS}^2 \le C |x_1 -
    x_2 |^2.
  \end{align}
  Consequently, the mapping $h^{\frac{1}{2}} g \circ F_h^{-1} \colon \R^m \to
  \R^m$ is also of linear growth and we conclude as above
  \begin{align*}
    \| g( U_h^{n} ) \Delta_h W^{n+1} \|^2 = 
    h \triple g(U^n_h) \triple^2
    &= h \triple g( F_h^{-1}( R_h^n )) \triple^2
    \le C \big( 1 + \| R_h^n \|^2 \big).
  \end{align*}
  In particular, this gives $g( U_h^{n} ) \in
  L^2(\Omega,\F_{t_n},\P;\R^{m \times d})$. 
  Finally, from the definition of $F_h$ and \eqref{eq:RDEinverse} we have that
  \begin{align*}
    f(U_h^n) = \frac{1}{\beta_k h}( U_h^n  + \beta_k h f(U_h^n) - U_h^n )
    = \frac{1}{\beta_k h} ( U_h^n - F_h (U_h^n) ) = \frac{1}{\beta_k h}(
    F_h^{-1}( R_h^n ) - R_h^n ).
  \end{align*}
  Hence, by the linear growth of $F_h^{-1}$ and the fact that $R_h^n \in
  L^2(\Omega,\mathcal{F}_{t_n},\P;\R^m)$ we conclude that $f(U^n_h) \in
  L^2(\Omega,\mathcal{F}_{t_n},\P;\R^m)$. 
  
  Altogether, this proves that $(U_h^{j,n+1})_{j \in \{0,\ldots,N_h\}}
  \in\mathcal{G}_h^2$ and, therefore, 
  by induction $(U_h^{j,n})_{j \in \{0,\ldots,N_h\}} \in\mathcal{G}_h^2$ for
  every $n\in\{k,\dots,N_h+1\}$. By finally noting that $U_h^{j}=U_h^{j,N_h+1}$,
  $j\in\{0,\dots,N_h\}$, the proof is complete. \qed
\end{proof}

\section{The backward Euler-Maruyama method}
\label{sec:BEM}

In this section we prove that the backward Euler-Maruyama scheme is mean-square
convergent of order $\frac{1}{2}$ under Assumption~\ref{as:fg}. The proof is
split over several subsections: First we familiarize ourselves with the
connection between the BEM method and the identity \eqref{eq:emmrich1}. This is
done by proving an a priori estimate in Subsection~\ref{subsec:BEMapriori}. 
In Subsection~\ref{subsec:BEMstability} we then derive a stability result
which gives an estimate of the distance between an arbitrary adapted grid
function and the one generated by the BEM method. As it turns out this distance
is bounded by the error in the initial value and a local truncation error.
The latter is estimated for the restriction of the exact solution to
\eqref{eq:SDE} to the temporal grid $\tau_h$ in
Subsection~\ref{subsec:BEMcons}. Altogether, this will then yield the 
desired convergence result in Section~\ref{subsec:BEMconvergence}.

\subsection{Basic properties of the backward Euler-Maruyama scheme}
\label{subsec:BEMapriori}

Here and in Subsection~\ref{subsec:BEMstability} we study
$U\in\mathcal{G}_h^2$, $h\in(0,\tfrac1L)$, satisfying
\begin{align}
  \label{eq:BEM3}
  U^j = U^{j-1} + h f(U^{j}) + g(U^{j-1}) \Delta_h W^{j}, \quad j \in
  \{1,\ldots, N_h\},
\end{align}
with initial condition $U^0\in L^2(\Omega,\mathcal{F}_0,\P;\R^m)$ such that
$f(U^0)\in L^2(\Omega,\mathcal{F}_0,\P;\R^m)$ and $g(U^0)\in
L^2(\Omega,\mathcal{F}_0,\P;\R^{m\times d})$. Here $L$ is the parameter in
Assumption~\ref{as:fg} and from Theorem~\ref{thm:wellposedness} there exist for
every $h\in(0,\tfrac1L)$ a unique  $U\in\mathcal{G}_h^2$ satisfying
\eqref{eq:BEM3}. $U_0$ is not necessarily related 
to the initial value $X_0$ of \eqref{eq:SDE}. 

In order to prove the a priori bound of Theorem~\ref{thm:Eulerapriori} and the
stability in Theorem~\ref{thm:BEMstability} the following lemma is used: 

\begin{lemma}
  \label{lemma:Euler}
  For all $h\in(0,\tfrac1L)$ and $U,V\in \mathcal{G}_h^2$ with $U$ satisfying
  \eqref{eq:BEM3}  it holds for all $j \in \{1,\ldots,N_h\}$ $\P$-almost
  surely that    
  \begin{align*}
    & |E^j|^2 - |E^{j-1}|^2 + |E^j-E^{j-1}|^2\\
    &\quad = 
    2h \big( f(U^j), E^j \big)
    + 2
    \big(
      g(U^{j-1})\Delta_h W^{j} , E^j - E^{j-1} \big)
    -2 \big( V^j - V^{j-1}, E^j \big)
    +
    Z^j,
  \end{align*}
  where $E:=U-V$ and $(Z^j)_{j\in\{1,\dots N_h\}}$
  are the centered random variables given by
  \begin{align*}
    Z^j
    := 2
    \big(
      g(U^{j-1})\Delta_hW^{j},E^{j-1}
    \big).
  \end{align*}
\end{lemma}

\begin{proof}
  From the identity \eqref{eq:emmrich1}, and since $U$ satisfies
  \eqref{eq:BEM3} by assumption the assertion follows directly. 
  Note that $Z^j$ is well-defined as a centered real-valued integrable random
  variable due to the independence of the centered Wiener increment
  $\Delta_h W^{j}$ and the square integrable random variables $g(U^{j-1})$ and
  $E^{j-1}$. \qed 
\end{proof}

The proof of the next theorem is the first and simplest demonstration of the,
in principle, same technique used to prove Theorems~\ref{thm:BEMstability},
\ref{thm:apriori} and \ref{thm:stability} below. This a priori estimate is in
fact not needed further in the analysis and it can be deduced from the
stability Theorem~\ref{thm:BEMstability}, but with larger constants, and for a more narrow range for the parameter $h$. We include it for completeness.

\begin{theorem}
  \label{thm:Eulerapriori}
  Let Assumption~\ref{as:fg} hold with $L\in(0,\infty)$, 
  $q\in[1,\infty)$. For $h\in(0,\tfrac1{2L})$ denote by 
  $U\in\mathcal{G}_h^2$ the unique adapted grid function satisfying
  \eqref{eq:BEM3}. Then, for all $n \in \{1,\ldots,N_h\}$ it holds that 
  \begin{align*} 
    &\|U^n\|^2 + h \triple g(U^n) \triple^2 
    \leq 
    C_h
    \exp\Big( \frac{2L t_n}{1-2Lh}\Big)
    \Big(
    1
    +
     \|U^{0}\|^2
    + h
    \triple
      g(U^0)
    \triple^2
    \Big),
  \end{align*}
  where $C_h= \max\{1,2LT\}(1-2Lh)^{-1}$.
\end{theorem}  
  
\begin{proof}
  Lemma~\ref{lemma:Euler} applied with $V=0$ and taking expectations yields  
  \begin{align*}
    & \|U^j\|^2 - \|U^{j-1}\|^2
    + \|U^j-U^{j-1}\|^2 \\
    &\quad = 2h \big\langle f(U^j), U^j \big\rangle
    + 2
    \big\langle
      g(U^{j-1})\Delta_h W^{j}, U^j - U^{j-1} \big\rangle.
  \end{align*}
  From the coercivity condition \eqref{eq:bound_fg} and the Young inequality
  \eqref{eq:Young} we have that  
  \begin{align*}
    & \|U^j\|^2 - \|U^{j-1}\|^2
    \leq 2h
    \Big(
      L \big( 1 + \|U^j\|^2 \big)
      -
      \frac{4q-3}2\triple g(U^j) \triple^2
      \Big)
    + h
    \triple
      g(U^{j-1})
    \triple^2.
  \end{align*}
  Summing over $j$ from $1$ to   $n$ gives
  that 
  \begin{align*}
    & 
    \big(
     1 - 2Lh
    \big) 
    \|U^n\|^2 + (4q - 3)h \triple g(U^n) \triple^2 \\    
    &\quad \leq 2 LT + \|U^0 \|^2 +   h \triple g(U^0) \triple^2
    + (4-4q) h \sum_{j = 1}^{n-1} \triple g(U^{j}) \triple^2
    + 2 L h \sum_{j=1}^{n-1} \|U^j\|^2.
  \end{align*}
  Since $q \in [1, \infty)$ it holds $1 \le 4q - 3$ and $4 - 4q \le 0$. By
  elementary bounds we get 
  \begin{align*} 
    &\|U^n\|^2 + h \triple g(U^n) \triple^2 \leq  
    \frac{2LT + \|U^{0}\|^2 + h
    \triple g(U^0) \triple^2}{1 - 2Lh} + \frac{2 L}{1 - 2Lh}h \sum_{j=1}^{n-1} \| U^j \|^2.
  \end{align*}
  We conclude by a use of the discrete Gronwall Lemma \eqref{eq:Gronwall}.
  \qed
\end{proof}

\subsection{Stability of the backward Euler-Maruyama scheme}
\label{subsec:BEMstability}

For the formulation of the stability Theorem~\ref{thm:BEMstability} we define
for $h\in(0,1)$ and $V\in\mathcal{G}_h^2$ the \emph{local
truncation error} of $V$ given by 
\begin{align}
  \label{eq:rho_Eul}
  \begin{split}
    \rho_h^{\mathrm{BEM}}(V) 
    &:= \sum_{j = 1}^{N_h} \big\| \varrho_{h}^j(V) \big\|^2
    + \frac{1}{h} \sum_{j = 1}^{N_h} \big\| P_h^{j-1} \varrho_h^j(V) 
    \big\|^2,
  \end{split}
\end{align}
where the \emph{local residuals} $\varrho^j_h(V)$ of $V$ are defined as
\begin{align*}
  \varrho^j_h(V) := h f(V^j) + g(V^{j-1}) \Delta_h W^j - V^{j} + V^{j-1},
\end{align*}
for $j\in\{2,\ldots,N_h\}$. Note that $\varrho^j_h(V) \in
L^2(\Omega,\F_{t_j},\P;\R^m)$ for every $V \in \mathcal{G}_h^2$. We also
introduce a maximal step size $h_E$ for the stability, which guarantees that
the stability constant in Theorem~\ref{thm:BEMstability} does not depend on
$h$. It is given by 
\begin{align}\label{eq:h0}
  h_E =\frac{1}{ \max\{4L,2\} }.
\end{align}
Using the same arguments as in Theorem~\ref{thm:Eulerapriori} the assertion of
Theorem~\ref{thm:BEMstability} stays true for all $h \in
(0,\frac{1}{2L})$ but with a constant $C$ depending on $\frac{1}{1 - 2hL}$ as
in Theorem~\ref{thm:Eulerapriori}.

\begin{theorem}
  \label{thm:BEMstability}
  Let Assumption~\ref{as:fg} hold with $L\in(0,\infty)$, 
  $\eta \in (\tfrac12, \infty)$. For all $h\in(0,h_E]$, $U\in \mathcal{G}_h^2$
  satisfying 
  \eqref{eq:BEM3}, $V\in \mathcal{G}_h^2$, and all $n
  \in \{1,\ldots,N_h\}$ it holds that  
  \begin{align*}
    &\|U^n - V^n \|^2 + h \triple  g(U^{n}) - g(V^{n}) \triple^2\\
    &\quad \le C \exp \big( 2 (1+2L) t_n \big)
    \Big( \|U^0 - V^0\|^2 + h \triple g(U^{0}) - g(V^{0}) \triple^2 +
    \rho_h^{\mathrm{BEM}}(V) \Big), 
  \end{align*}
  where $C= \max\{ 3 , 4 \eta, \frac{4 \eta}{2 \eta - 1} \}$.  
\end{theorem}

\begin{proof}
  Fix arbitrary $h \in (0, h_E]$ and $V \in \mathcal{G}^2_h$. To ease the
  notation we suppress the dependence of $h$ and $V$ and simply write, for
  instance, $\Delta W^j := \Delta_h W^j$. We also write $E^j := U^j - V^j$ and
  \begin{align}\label{eq:Delta_fg}
    \Delta f^j:=f(U^j)-f(V^j)
    ,\quad
    \Delta g^j:=g(U^j)-g(V^j)
    ,\quad
    j\in\{0,\dots, N_h\}.
  \end{align}
  From Lemma~\ref{lemma:Euler} we get after
  taking expectations that
  \begin{align*}
    &  \| E^j \|^2 - \| E^{j-1} \|^2
    + \|E^j- E^{j-1}\|^2\\
    &\quad = 2 h \big\langle \Delta f^j,E^j\big\rangle
    + 2 \big\langle \Delta g^{j-1}\Delta W^j + \varrho^j, E^j-E^{j-1}
    \big\rangle + 2 \langle \varrho^j,E^{j-1} \rangle.
  \end{align*}
  In order to treat the residual term we first notice that
  $P_h^{j-1} E^{j-1}=E^{j-1}$. Then, by taking the adjoint of the projector
  and by applying the weighted Young inequality \eqref{eq:Young} with $\nu=h
  >0$ we obtain 
  \begin{align*}
    2 \big\langle \varrho^j, E^{j-1} \big\rangle 
    &= 2\big\langle P_h^{j-1} \varrho^j, E^{j-1} \big\rangle
    \le \frac{1}{h} \big\| P_h^{j-1} \varrho^j \big\|^2
    + h \big\| E^{j-1} \big\|^2.
  \end{align*}
  Moreover, further applications of the Cauchy-Schwarz inequality, the triangle
  inequality, and the weighted Young inequality \eqref{eq:Young} with $\nu =
  \mu$ yield 
  \begin{align*}
    &2 \big\langle \Delta g^{j-1}\Delta W^j + \varrho^j, E^j-E^{j-1}
    \big\rangle\\
    &\quad \le \big\| \Delta g^{j-1}\Delta W^j + \varrho^j \big\|^2
    + \big\| E^j-E^{j-1} \big\|^2 \\
    &\quad \le \big( 1 + \mu \big) \big\| \Delta g^{j-1}\Delta W^j \big\|^2
    + \Big(1 + \frac{1}{\mu} \Big) \big\| \varrho^j \big\|^2 
    + \big\| E^j - E^{j-1} \big\|^2. 
  \end{align*}
  Therefore, together with the global monotonicity condition
  \eqref{eq:onesided} this gives  
  \begin{align*}
    \|E^j\|^2 - \|E^{j-1}\|^2
    & \leq 2h L \big\| E^j \big\|^2
    - 2 h \eta \triple \Delta g^j \triple^2
    + (1 + \mu) h \triple \Delta g^{j-1}\triple^2\\
    &\quad + h \big\| E^{j-1} \big\|^2 
    + \Big(1 + \frac1\mu\Big) \big\| \varrho^j \big\|^2 
    + \frac{1}{h} \big\| P_h^{j-1} \varrho^j \big\|^2. 
  \end{align*}
  Setting $\mu= 2\eta-1 >0$ gives that $1 + \mu = 2 \eta$. Then, summing
  over $j$ from $1$ to $n$ and thereby identifying two telescoping sums yields
  \begin{align*}
    & \big(1-2Lh\big)\|E^n\|^2 + 2 \eta h \triple \Delta g^n \triple^2 \\ 
    &\quad \leq (1+h)\|E^{0}\|^2 + 2 \eta h \triple \Delta g^{0}\triple^2
    + (1 + 2 L ) h \sum_{j=1}^{n-1}\big\| E^j \big\|^2 \\
    &\qquad + \frac{2\eta}{2\eta-1} \sum_{j=1}^n\big\| \varrho^j \big\|^2
    + \frac{1}{h} \sum_{j=1}^n\big\| P_h^{j-1} \varrho^j 
    \big\|^2.
  \end{align*}
  Since $1 - 2 Lh \ge 1 - 2 L h_E > \frac{1}{2}$ as well as $h\le h_E<\tfrac12$
  and $\eta>\tfrac12$ we obtain after some elementary transformations the
  inequality 
  \begin{align*}
    \|E^n\|^2 + h \triple \Delta g^n \triple^2 
    &\leq \max\Big\{ 3 , 4 \eta, \frac{4 \eta}{2\eta-1}\Big\}
    \Big(\|E^0\|^2 + h \triple \Delta g^{0} \triple^2 
    +  \rho_h^{\mathrm{BEM}}(V)\Big)\\
    &\qquad + 2  (1+2L)  h \sum_{j=1}^{n-1} \big\| E^j \big\|^2.
  \end{align*}
  The proof is completed by applying the discrete Gronwall Lemma
  \eqref{eq:Gronwall}. \qed
\end{proof}

\subsection{Consistency of the backward Euler-Maruyama scheme}
\label{subsec:BEMcons}
In this subsection we give an estimate for the local truncation error
\eqref{eq:rho_Eul} of the BEM method. 
For the proof we first recall that the restriction $X|_{\tau_h}$ of the
exact solution to the temporal grid $\tau_h$ is an element of the space
$\mathcal{G}_h^2$, see Subsection~\ref{subsec:setting}.
Further, we make use of \cite[Lemma~5.5,
Lemma~5.6]{beynisaakkruse2014}, which provide estimates for the drift integral 
\begin{equation}
  \label{eq:bound_f}
  \begin{split}
    & \int_{\tau_1}^{\tau_2} \big\| f(X(\tau)) - f(X(s)) \big\| \diff{s}
    \leq C \Big( 1 + \sup_{t\in[0,T]} \big\| X(t)
    \big\|_{L^{4q-2}(\Omega;\R^m)}^{2q-1} \Big) |\tau_2-\tau_1|^{\frac32},
  \end{split}
\end{equation}
for all $\tau, \tau_1,\tau_2\in [0,T]$ with $\tau_1 \le \tau \le \tau_2$, 
and for the stochastic integral
\begin{equation}
  \label{eq:bound_g}
  \begin{split}
    & \Big\| \int_{\tau_1}^{\tau_2} \big( g(X(\tau_1)) - g(X(s)) \big)
    \diff{W(s)} \Big\| \leq C \Big( 1 +  \sup_{t\in[0,T]} \big\| X(t)
    \big\|_{L^{4q-2}(\Omega;\R^m)}^{2q-1} \Big) |\tau_2-\tau_1|,
  \end{split}
\end{equation}
for all $\tau_1,\tau_2\in [0,T]$ with $\tau_1 \le \tau_2$, respectively.

\begin{theorem}
  \label{thm:BEMconsistency}
  Let Assumption~\ref{as:fg} hold and let $X|_{\tau_h}$ be the restriction of
  the exact solution to \eqref{eq:SDE} to the temporal grid $\tau_h$. Then there
  exists $C>0$ such that  
  \begin{align*} 
    \rho_h^\mathrm{BEM}(X|_{\tau_h}) \leq C h, \quad h\in(0,1), 
  \end{align*}
  where the local truncation is defined in \eqref{eq:rho_Eul}.
\end{theorem}

\begin{proof}
  Recall the definitions of $\rho_h^\mathrm{BEM}(X|_{\tau_h})$ and
  $\varrho_h^j(X|_{\tau_h})$ from \eqref{eq:rho_Eul}. It suffices to show 
  \begin{equation}
    \label{eq:BEMconv1}
    \begin{split}
      \max_{j\in\{1,\dots,N_h\}}
      \Big( \big\| \varrho^j_h(X|_{\tau_h}) \big\|^2 +
      \frac{1}{h} \big\|P_h^{j-1}\varrho^j_h(X|_{\tau_h})\big\|^2 \Big) 
      \leq C h^2.
    \end{split}
  \end{equation}
  Inserting \eqref{eq:SDE} it holds for every $j\in\{1,\dots,N_h\}$ that
  \begin{align*}
    \varrho^j_h(X|_{\tau_h}) &= 
    \int_{t_{j-1}}^{t_j} \big( f(X(t_j)) - f(X(s)) \big) \diff{s}\\
    &\quad + \int_{t_{j-1}}^{t_j} \big( g(X(t_{j-1})) - g(X(s)) \big)
    \diff{W(s)}\\ 
    P_h^{j-1} \varrho^j_h(X|_{\tau_h}) &= 
    \E \Big[ \int_{t_{j-1}}^{t_j} \big( f(X(t_{j})) - f(X(s)) \big) \diff{s}
    \big| \F_{t_{j-1}} \Big].
  \end{align*}
  Note that inequalities \eqref{eq:bound_f} and \eqref{eq:bound_g} apply 
  to \eqref{eq:BEMconv1} due to the moment bound \eqref{eq:moment}. These
  estimates together with an application of the triangle inequality and the
  fact that $\| \E [ V | \F_{t_{j-1}} ] \| \le \| V \|$ for every $V \in \LOR$ 
  completes the proof of \eqref{eq:BEMconv1}. 
  \qed
\end{proof}

\subsection{Mean-square convergence of the backward Euler-Maruyama method}
\label{subsec:BEMconvergence}

Here we consider the numerical approximations $(X_h^j)_{j=0}^{N_h}$,
$h\in(0,h_E]$, uniquely determined by the backward Euler-Maruyama method
\eqref{eq:BEM} with a corresponding family of initial values
$(X_h^0)_{h\in(0,h_E]}$. Recall from \eqref{eq:h0} that $h_E=\tfrac1{2(4L+1)}$.
This family is assumed to satisfy the following assumption.  
\begin{assumption}\label{as:BEMinitial}
The family of initial values $(X_h^0)_{h\in(0,h_E]}$ satisfies
\begin{align}
  \label{eq:BEMinitial_cons2}
  X_h^0, f(X_h^0) 
  \in
  L^2(\Omega,\mathcal{F}_0,\P;\R^m)
  ,\quad
   g(X_h^0)\in 
  L^2(\Omega,\mathcal{F}_0,\P;\R^{m\times d}),
\end{align}
for all $h\in(0,h_E]$ and is consistent of order $\tfrac12$ in the sense that 
\begin{align}
  \label{eq:BEMinitial_cons}
  \|X(0)-X_h^0\|^2
  +
  h\triple g(X(0)) - g(X_h^0) \triple^2
  =\mathcal{O}(h),
\end{align}
as $h\downarrow0$, where $X$ is the exact solution to \eqref{eq:SDE}. 
\end{assumption}

Note that Assumption~\ref{as:BEMinitial} is obviously satisfied for the choice
$X_h^0 := X_0$ for every $h \in (0,h_E]$. This said we are now ready to state the
main result of this section. 

\begin{theorem}
  \label{thm:BEMconvergence}
  Let Assumptions~\ref{as:fg} and \ref{as:BEMinitial} hold, let $X$ be the
  exact solution to \eqref{eq:SDE} and let $(X_h^j)_{j=0}^{N_h}$, $h \in
  (0,h_E]$, be the family of backward Euler-Maruyama approximations determined by
  \eqref{eq:BEM} with initial values $(X_h^0)_{h\in(0,h_E]}$. Then, the 
  backward Euler-Maruyama method is mean-square convergent of order $\tfrac12$,
  more precisely, there exists $C>0$ such that 
  \begin{align*}
    \max_{n\in\{0,\dots,N_h\}}
    \|X_h^n-X(nh)\| \leq C\sqrt{h} ,\quad h\in(0,h_E].
  \end{align*}
\end{theorem}

\begin{proof}
  For $h\in(0,h_E]$, we apply
  Theorem~\ref{thm:BEMstability} with $U=(X_h^j)_{j=0}^{N_h}\in\mathcal{G}_h^2$
  and $V=X|_{\tau_h}=(X(t_j))_{j=0}^{N_h}\in\mathcal{G}_h^2$ and get that there
  is a constant $C>0$, not depending on $h$, such that 
  \begin{align*}
    \| X_h^n - X(nh)\|^2
    &\le C \big( \| X_h^0 - X(0) \|^2 + h \triple g(X_h^{0}) - g(X(0))
    \triple^2 + \rho_h^\mathrm{BEM}(X|_{\tau_h}) \big).
  \end{align*}
  The first and second term on the right hand side are of order
  $\mathcal{O}(h)$ by 
  Assumption~\ref{as:BEMinitial}. Since the same holds true for the consistency
  term $\rho_h^\mathrm{BEM}(X|_{\tau_h})$ by Theorem~\ref{thm:BEMconsistency}
  the proof is completed. \qed
\end{proof}

\section{The BDF2-Maruyama method}
\label{sec:BDF2}
In this section we follow the same procedure as in Section~\ref{sec:BEM} with
identity \eqref{eq:emmrich2} in place of \eqref{eq:emmrich1}. Every
result in Section~\ref{sec:BEM} has its counterpart here for the
BDF2-Maruyama method. As the multi-step method involves more terms the proofs
in this section are naturally a bit more technical, but rely in principle on the
same arguments as in the previous section.

\subsection{Basic properties of the BDF2-Maruyama method}
\label{subsec:existence}
Here and in Subsection~\ref{subsec:stability} our results concern
$U\in\mathcal{G}_h^2$, $h\in(0,\tfrac3{2L})$, satisfying 
\begin{equation}
  \label{eq:BDF2U}
  \begin{split}
    \frac{3}{2} U^{j} - 2 U^{j-1} + \frac{1}{2} U^{j-2} 
    &= h f(U^j) + \frac{3}{2} g( U^{j-1} )
    \Delta_h W^j\\ 
  &\quad- \frac{1}{2}  g(U^{j-2} ) \Delta_h W^{j-1},
    \quad j\in\{2,\dots,N_h\},\\
  \end{split}
\end{equation}
with initial values $U^\ell \in
L^{2}(\Omega,\mathcal{F}_{t_\ell},\P;\R^m)$ such that $f(U^\ell) \in
L^{2}(\Omega,\mathcal{F}_{t_\ell},\P;\R^m)$ and $g(U^\ell) \in
L^{2}(\Omega,\mathcal{F}_{t_\ell},\P;\R^{m\times d})$ for $\ell\in\{0,1\}$.
Here $L$ is the parameter of Assumption~\ref{as:fg} and from
Theorem~\ref{thm:wellposedness} there exists for every $h\in(0,\tfrac{3}{2L})$ a
unique $U\in\mathcal{G}_h^2$ satisfying \eqref{eq:BDF2U}. The initial values
$(U^0,U^1)$ are not necessarily related to the initial value $X_0$ of
\eqref{eq:SDE}.

Next, we state an analogue of Lemma~\ref{lemma:Euler}, used for the proof of
the a priori estimate in Theorem~\ref{thm:apriori} and the stability result in
Theorem~\ref{thm:stability}.

\begin{lemma}\label{lemma:emmrich}
  For all $h\in(0,\tfrac3{2L})$ and $U,V\in \mathcal{G}_h^2$ with $U$ satisfying
  \eqref{eq:BDF2U}  it holds for all $j \in \{2,\ldots,N_h\}$ $\P$-almost
  surely that    
  \begin{align*}
    & |E^j|^2 - |E^{j-1}|^2
    + |2E^j-E^{j-1}|^2
    - |2E^{j-1}-E^{j-2}|^2
    + |E^j-2E^{j-1}+E^{j-2}|^2\\
    & = 
    4h \big( f(U^j), E^j \big)
    + 2
    \big(
      g(U^{j-1})\Delta_h W^{j} - g(U^{j-2})\Delta_h W^{j-1} , E^j - 2E^{j-1} +
      E^{j-2} \big)\\
     & \quad
     + 2
    \big(
      g(U^{j-1})\Delta_h W^j , 2E^j-E^{j-1}
    \big)
     - 2
    \big(
      g(U^{j-2})\Delta_h W^{j-1} , 2E^{j-1}-E^{j-2}
    \big)\\
    & \quad 
    -4 \Big( \frac32 V^j - 2 V^{j-1} + \frac12 V^{j-2}, E^j
    \Big)
    +
    Z^j,
  \end{align*}
  where $E:=U-V$ and $(Z^j)_{j\in\{2,\dots N_h\}}$,
  are the centered random variables given by
  \begin{align*}
    Z^j
    :=
    \big(
      g(U^{j-1})\Delta_hW^{j},6E^{j-1}-2E^{j-2}
    \big).
  \end{align*}
\end{lemma}

\begin{proof}
>From the identity \eqref{eq:emmrich2} and since $U$ satisfy \eqref{eq:BDF2U} by
assumption it holds for $j\in\{2,\dots,N_h\}$ that
  \begin{align*}
    & |E^j|^2 - |E^{j-1}|^2
    + |2E^j-E^{j-1}|^2
    - |2E^{j-1}-E^{j-2}|^2
    + |E^j-2E^{j-1}+E^{j-2}|^2\\
    & = 4 \Big( \frac32 E^j - 2 E^{j-1} + \frac12 E^{j-2}, E^j
    \Big)\\ 
    & = 4 \Big( \frac32 U^j - 2 U^{j-1} + \frac12 U^{j-2}, E^j
    \Big)
    - 4 \Big( \frac32 V^j - 2 V^{j-1} + \frac12 V^{j-2}, E^j
    \Big)\\ 
    & = 4h \big( f(U^j), E^j \big)
    +
    6 
    \big(
      g(U^{j-1})\Delta_h W^j , E^j
    \big)
    - 2
    \big(
      g(U^{j-2})\Delta_h W^{j-1} , E^j
    \big)\\
    &\quad-4 \Big( \frac32 V^j - 2 V^{j-1} + \frac12 V^{j-2}, E^j
    \Big).
  \end{align*}
  Adding, subtracting and rearranging terms completes the proof of the asserted
  identity. Further note that $Z^j$ is centered due to the independence of the
  centered Wiener increment $\Delta_h W^{j}$ from $g(U^{j-1})$, $E^{j-1}$, and
  $E^{j-2}$. \qed
\end{proof}
  
\begin{theorem}
  \label{thm:apriori}
  Let Assumption~\ref{as:fg} hold with $L\in(0,\infty)$, 
  $\eta \in [\tfrac12, \infty)$, $q\in[1,\infty)$. For
  $h\in(0,\tfrac1{4L})$ denote by $U\in\mathcal{G}_h^2$ the unique adapted grid
  function satisfying \eqref{eq:BDF2U}. Then, for all 
  $n \in \{2,\ldots,N_h\}$ it holds that 
  \begin{align*} 
    &\|U^n\|^2 + h \triple g(U^n) \triple^2\\
    &\quad \leq 
    C_h
    \exp\Big(\frac{4Lt_n}{1-4Lh}\Big)
    \Bigg(
    1
    +
     \|U^{1}\|^2
    +
    \|2U^{1}-U^{0}\|   ^2
    +
    \sum_{l=0}^1
    \triple
      g(U^l)
    \triple^2
    \Bigg),
  \end{align*}
  where $C_h=4\max\{1,LT\}(1-4Lh)^{-1}$.
\end{theorem}  
  
\begin{proof}
  Applying Lemma~\ref{lemma:emmrich} with $V=0$ and taking expectations yields  
  \begin{align*}
    & \|U^j\|^2 - \|U^{j-1}\|^2
    + \|2U^j-U^{j-1}\|^2
    - \|2U^{j-1}-U^{j-2}\|^2
    + \|U^j-2U^{j-1}+U^{j-2}\|^2\\
    & = 4h \big\langle f(U^j), U^j \big\rangle
    + 2
    \big\langle
      g(U^{j-1})\Delta_h W^{j} - g(U^{j-2})\Delta_h W^{j-1} , U^j - 2U^{j-1} +
      U^{j-2} 
    \big\rangle\\
     & \quad
      + 2
    \big\langle
      g(U^{j-1})\Delta_h W^j , 2U^j-U^{j-1}
    \big\rangle - 2
    \big\langle
      g(U^{j-2})\Delta_h W^{j-1} , 2U^{j-1} - U^{j-2}
    \big\rangle.
  \end{align*}
  From the Young inequality \eqref{eq:Young}, the orthogonality
  \begin{align*}
    \big\| g(U^{j-1}) \Delta_h W^j - g(U^{j-2})
    \Delta_h W^{j-1} \big\|^2 = h \triple g(U^{j-1}) \triple^2 + h \triple
    g(U^{j-2}) \triple^2, 
  \end{align*}
  and the coercivity condition \eqref{eq:bound_fg} we have that
  \begin{align*}
    & \|U^j\|^2 - \|U^{j-1}\|^2
    + \|2U^j-U^{j-1}\|^2
    - \|2U^{j-1}-U^{j-2}\|^2\\
    &\quad \leq 4h
    \Big(
      L \big( 1 + \|U^j\|^2 \big)
      -
      \frac{4q-3}2\triple g(U^j) \triple^2
      \Big)
    + h
    \triple
      g(U^{j-1})
    \triple^2    
    + h
    \triple
      g(U^{j-2})
    \triple^2 \\
     &\qquad + 2
    \big\langle
      g(U^{j-1})\Delta_h W^j , 2U^j - U^{j-1}
    \big\rangle
    - 2
    \big\langle
      g(U^{j-2})\Delta_h W^{j-1} , 2U^{j-1} - U^{j-2}
    \big\rangle.
  \end{align*}
  Summing over $j$ from $2$ to   $n$, identifying three telescoping sums, using
  the Young inequality \eqref{eq:Young} gives that 
  \begin{align*}
    & \|U^n\|^2
    + \|2U^n-U^{n-1}\|^2\\    
    & \leq 
    4 LT +
    \|U^{1}\|^2
    +
    \|2U^{1}-U^{0}\|^2
    + 
    4Lh \sum_{j=2}^n
    \|U^j\|^2
    -
      (8q-6)h\sum_{j=2}^n\triple g(U^j) \triple^2\\
    &\quad
    +   
    h
    \sum_{j=1}^{n-1}
    \triple
      g(U^j)
    \triple^2
    +
    h
    \sum_{j=0}^{n-2}
    \triple
      g(U^j)
    \triple^2
    + 
    h
    \triple
      g(U^{n-1})
    \triple^2 
    +
    \big\|
      2U^n-U^{n-1}
    \big\|^2\\
    &    
    \quad + h
    \triple
      g(U^{0})
    \triple^2
    +
    \big\|
      2U^1-U^0
    \big\|^2.
  \end{align*}
  This yields
  \begin{align*}
    & 
    \big(
     1 - 4Lh
    \big) 
    \|U^n\|^2
    + (8q - 6)h \triple g(U^n) \triple^2 \\    
    &\quad \leq 
     4LT + \|U^{1}\|^2
    +
    2\|2U^{1}-U^{0}\|^2 
    +   2 h
    \sum_{l=0}^1
    \triple
      g(U^l)
    \triple^2\\
    & \qquad
    + (8 - 8q ) h \sum_{j=2}^{n-1} \triple g(U^j) \triple^2
    + 4Lh \sum_{j=2}^{n-1}
    \|U^j\|^2.
  \end{align*}
  Since $q \in [1, \infty)$ it holds $8 - 8 q \le 0$ and $8q-6\geq 2$. By
  elementary bounds we get
  \begin{align*} 
    & \|U^n\|^2 + h \triple g(U^n) \triple^2
     \leq 
    \frac{4L}{1-4Lh} h \sum_{j=2}^{n-1}
    \|U^j\|^2 \\
    &\qquad
     +\frac1{1-4Lh}
     \Big( 4LT +
     \|U^{1}\|^2
    +
    2\|2U^{1}-U^{0}\|^2
    + 
    2h \sum_{l=0}^1
    \triple
      g(U^l)
    \triple^2 \Big).
  \end{align*}
  We conclude by a use of the discrete Gronwall Lemma \eqref{eq:Gronwall}.
  \qed
\end{proof}

\subsection{Stability of the BDF2-Maruyama scheme}
\label{subsec:stability}

Similar to the stability of the BEM scheme, for $h\in(0,1)$,
$V\in\mathcal{G}_h^2$ we define the local truncation error of $V$ by 
\begin{align}
  \label{eq:rho_h}
  \begin{split}
    \rho_h^\mathrm{BDF2}(V) 
    &:= 
    \max_{j\in  \{1,\dots,N_h\}} \frac{1}{h} \big\| \rho_{1,h}^j(V)
    \big\|^2
    +   
    \sum_{j = 2}^{N_h}
    \big\| \rho_{2,h}^j(V) 
    +
    \rho_{3,h}^j(V)   \big\|^2\\
    &\quad
    + \frac{1}{h} \sum_{j = 2}^{N_h}\big\| P_h^{j-1} \rho_{2,h}^j(V) 
    \big\|^2+ \frac{1}{h} \sum_{j = 2}^{N_h}\big\| P_h^{j-2}
    \rho_{3,h}^j(V) \big\|^2,
  \end{split}
\end{align}
where
\begin{align}
  \label{eq:rho_i}
  \begin{split}
    \rho_{1,h}^j(V) &:= h f(V^j) + g(V^{j-1}) \Delta_h W^j - V^j + V^{j-1},\\
    \rho_{2,h}^j(V) &:= \frac{1}{2} \big( h f(V^{j-1}) + g(V^{j-1}) \Delta_h
    W^j - V^j + V^{j-1} \big),\\
    \rho_{3,h}^j(V) &:= -\frac{1}{2} \big( h f(V^{j-1}) + g(V^{j-2}) \Delta_h 
    W^{j-1} - V^{j-1} + V^{j-2} \big), 
  \end{split}
\end{align}
for $j\in\{2,\dots,N_h\}$. Similar to \eqref{eq:h0} we define the maximal step
size 
\begin{align*}
  h_B=\frac1{2 (4L + 1) }.  
\end{align*}
Note that the proof of Theorem~\ref{thm:apriori} indicates that the assertion
of the following theorem actually holds true for all $h \in (0,
\frac1{4L})$ but the constants on the right hand side then depend on
$\frac{1}{1 - 4 h L}$.

\begin{theorem}
  \label{thm:stability}
  Let Assumption~\ref{as:fg} hold with $L\in(0,\infty)$, 
  $\eta \in (\tfrac12, \infty)$. For
  all $h\in(0,h_B]$, $U\in \mathcal{G}_h^2$ satisfying \eqref{eq:BDF2U},
  $V\in \mathcal{G}_h^2$, and all $n \in \{2,\ldots,N_h\}$ it holds that  
  \begin{align*}
    &\|U^n - V^n\|^2 + h \triple g(U^n) - g(V^n) \triple^2\\
    &\quad \le C \exp\Big(4 \max\Big\{  (1+2L)  ,  \frac{\eta}{2\eta-1} \Big\}
    t_n \Big)\\ 
    & \qquad \times\Big( \sum_{\ell=0}^1
    \big( \|U^\ell - V^\ell\|^2 + h
    \triple g(U^{\ell}) - g(V^{\ell}) \triple^2 \big)
    +\rho_h^\mathrm{BDF2}(V) \Big),
  \end{align*}
  where $C= \max\{30, 4\eta + 2, \tfrac{16 \eta}{2\eta-1}\}$.
\end{theorem}

\begin{proof}
  Fix arbitrary $h \in (0, h_B]$ and $V \in \mathcal{G}^2_h$. We reuse the
  notation from the proof of Theorem~\ref{thm:BEMstability}. In particular we
  set $E := U - V$ and we often suppress $h$ from the notation. The local
  residual of $V$ is given by  
  \begin{align*}
    \varrho^j := h f(V^j) + \frac{3}{2} g(V^{j-1}) \Delta W^j -
    \frac{1}{2} g(V^{j-2}) \Delta W^{j-1} - \frac{3}{2} V^{j} + 2 V^{j-1} -
    \frac{1}{2} V^{j-2},
  \end{align*}
  for $j\in\{2,\ldots,N_h\}$. From Lemma~\ref{lemma:emmrich} we get after
  taking expectations that
  \begin{align*}
    &  \|E^j\|^2 - \|E^{j-1}\|^2
    + \|2E^j-E^{j-1}\|^2 
    - \|2E^{j-1}-E^{j-2}\|^2
    + \|E^j-2E^{j-1}+E^{j-2}\|^2\\
    & = 4\big\langle \Delta f^j,E^j\big\rangle
    + 2 \big\langle \Delta g^{j-1}\Delta W^j-\Delta g^{j-2}\Delta W^{j-1},
    E^j-2E^{j-1}+E^{j-2} \big\rangle
    + 4 \langle \varrho^j,E^j \rangle\\
    &  \quad + 2 \big\langle \Delta g^{j-1} \Delta W^j,
    2E^j - E^{j-1}\big\rangle - 2 \big\langle\Delta g^{j-2} \Delta W^{j-1},
    2E^{j-1} - E^{j-2}\big\rangle.
  \end{align*}
  We observe that
  $P_h^{j-1}(2E^{j-1}-E^{j-2})=2E^{j-1}-E^{j-2}$ and $P_h^{j-2} E^{j-2}=E^{j-2}$.
  Further, we decompose the local residual of $V$ by  
  \begin{align}
    \label{eq:rho}
      \varrho^j = \rho_1^j + \rho_2^j + \rho_3^j,
  \end{align} 
  where $\rho_i^j := \rho_i^j(V)$, $i \in \{1,2,3\}$, are defined in
  \eqref{eq:rho_i}. Then, by taking the adjoints of the projectors and by
  applying the weighted Young inequality \eqref{eq:Young} with $\nu=\mu>0$ and
  with $\nu=h >0$, respectively, and by noting that
  $-2\rho_3^j=\rho_1^{j-1}$, we obtain 
  \begin{align*}
    &\big\langle \varrho^j, E^j \big\rangle  = \big\langle \rho_1^j, E^j
    \big\rangle + \big\langle \rho_2^j + \rho_3^j, E^j - 2 E^{j-1} +
    E^{j-2} \big\rangle + \big\langle \rho_2^j + \rho_3^j, 2 E^{j-1} - E^{j-2}
    \big\rangle\\
    &\quad \le \big\langle \rho_1^j, E^j \big\rangle
    + \frac{\mu}2 \big\| \rho_2^j + \rho_3^j \big\|^2
    + \frac{1}{2\mu} \big\| E^j - 2 E^{j-1} + E^{j-2} \big\|^2 \\
    &\qquad + \big\langle  \rho_2^j, P_h^{j-1}\big(2 E^{j-1} - E^{j-2}\big)
    \big\rangle
    + \big\langle \rho_3^j, 2 E^{j-1} \big\rangle
    - \big\langle \rho_3^j , P_h^{j-2} E^{j-2} \big\rangle\\
    &\quad \le \big\langle \rho_1^j, E^j \big\rangle 
    - \big\langle \rho_1^{j-1}, E^{j-1} \big\rangle
    + \frac{\mu}{2} \big\| \rho_2^j + \rho_3^j \big\|^2
    + \frac{1}{2\mu} \big\| E^j - 2 E^{j-1} + E^j \big\|^2\\
    &\qquad + \frac{1}{2h} \big\| P_h^{j-1} \rho_2^j 
    \big\|^2 + \frac{h}{2} \big\| 2 E^{j-1} - E^{j-2} \big\|^2
    + \frac{1}{2h} \big\| P_h^{j-2} \rho_3^j 
    \big\|^2 + \frac{h}{2} \big\| E^{j-2} \big\|^2.
  \end{align*}
  Together with the global monotonicity condition \eqref{eq:onesided} and the
  weighted Young inequality \eqref{eq:Young} with $\nu=2\eta$ this gives that 
  \begin{align*}
    &  \|E^j\|^2 - \|E^{j-1}\|^2
    + \|2E^j-E^{j-1}\|^2 
    - \|2E^{j-1}-E^{j-2}\|^2\\ 
    & \leq 4h L \big\| E^j \big\|^2
    - 4 h \eta \triple \Delta g^j \triple^2
    + 2\eta h \triple \Delta g^{j-1}\triple^2
    + 2\eta h \triple \Delta g^{j-2}\triple^2\\
    & \quad+ \Big( \frac2\mu + \frac1{2\eta} - 1\Big)
    \big\| E^j - 2E^{j-1} + E^{j-2} \big\|^2
    + 2h \big\| 2 E^{j-1} - E^{j-2} \big\|^2
    + 2h \big\| E^{j-2} \big\|^2\\
    & \quad
    + 2 \big\langle \Delta g^{j-1} \Delta W^j,
    2E^j - E^{j-1}\big\rangle - 2 \big\langle\Delta g^{j-2} \Delta W^{j-1},
    2E^{j-1} - E^{j-2}\big\rangle\\
    & \quad + 4\big\langle \rho_1^j, E^j \big\rangle 
    - 4\big\langle \rho_1^{j-1}, E^{j-1} \big\rangle
    + 2\mu \big\| \rho_2^j + \rho_3^j \big\|^2
    + \frac{2}{h} \big\| P_h^{j-1} \rho_2^j 
    \big\|^2 
    + \frac{2}{h} \big\| P_h^{j-2} \rho_3^j 
    \big\|^2.
  \end{align*}
  Setting $\mu=\tfrac{4\eta}{2\eta-1}>0$ gives that $\tfrac2\mu +
  \tfrac1{2\eta} - 1=0$. Then, summing over $j$ from $2$ to $n$ and identifying
  four telescoping sums yields 
  \begin{align*} 
    &  \|E^n\|^2
    + \|2E^n-E^{n-1}\|^2\\ 
    & \leq \|E^{1}\|^2 + \|2E^{1}-E^{0}\|^2 
    + 4h L \sum_{j=2}^{n}\big\| E^j \big\|^2
    - 4 \eta h \sum_{j=2}^n\triple \Delta g^j \triple^2
    + 2\eta h \sum_{j=1}^{n-1}\triple \Delta g^{j}\triple^2\\
    &\quad + 2\eta h \sum_{j=0}^{n-2}\triple \Delta g^{j}\triple^2
    + 2h \sum_{j=1}^{n-1}\big\| 2 E^{j} - E^{j-1} \big\|^2
    + 2h \sum_{j=0}^{n-2}\big\| E^{j} \big\|^2\\
    & \quad
    + 2 \big\langle \Delta g^{n-1} \Delta W^n,
    2E^n - E^{n-1}\big\rangle - 2 \big\langle\Delta g^{0} \Delta W^{1},
    2E^{1} - E^0\big\rangle+ 4\big\langle \rho_1^n, E^n \big\rangle\\ 
    &\quad - 4\big\langle \rho_1^{1}, E^{1} \big\rangle
    + \frac{8\eta}{2\eta-1} \sum_{j=2}^n\big\| \rho_2^j + \rho_3^j \big\|^2
    + \frac{2}{h} \sum_{j=2}^n\big\| P_h^{j-1} \rho_2^j 
    \big\|^2 
    + \frac{2}{h} \sum_{j=2}^n\big\| P_h^{j-2} \rho_3^j 
    \big\|^2.
  \end{align*}
  Next, we get from the weighted Young inequality \eqref{eq:Young} with $\nu =
  \frac{2}{h}$ that $4\big\langle \rho_1^{n}, E^{n} \big\rangle \le \frac{4}{h}
  \| \rho_1^n \|^2 + h \|E^{n}\|^2$. A further application of the weighted
  Young inequality \eqref{eq:Young} with $\nu=2\eta$ yields
  \begin{align*}
    & \big(1-(4L+1)h\big)\|E^n\|^2
    + \Big( 1 - \frac1{2\eta} \Big)\|2E^n-E^{n-1}\|^2\\ 
    & \leq (1+h)\|E^{1}\|^2 + (1+h)\|2E^{1}-E^{0}\|^2 
    + 4h L \sum_{j=2}^{n-1}\big\| E^j \big\|^2
    - 4 \eta h \sum_{j=2}^n\triple \Delta g^j \triple^2\\
    & \quad + 2\eta h \sum_{j=1}^{n-1}\triple \Delta g^{j}\triple^2
    + 2\eta h \sum_{j=0}^{n-2}\triple \Delta g^{j}\triple^2
    + 2h \sum_{j=1}^{n-1}\big\| 2 E^{j} - E^{j-1} \big\|^2
    + 2h \sum_{j=0}^{n-2}\big\| E^{j} \big\|^2\\
    & \quad
    + 2\eta h \triple \Delta g^{n-1} \triple^2 
    + h \triple \Delta g^{0} \triple^2 + \frac4h\| \rho_1^n \|^2 
    + \frac4h\| \rho_1^{1}\|^2
    + \frac{8\eta}{2\eta-1} \sum_{j=2}^n\big\| \rho_2^j
    + \rho_3^j \big\|^2\\
    & \quad + \frac{2}{h} \sum_{j=2}^n\big\| P_h^{j-1} \rho_2^j 
    \big\|^2 
    + \frac{2}{h} \sum_{j=2}^n\big\| P_h^{j-2} \rho_3^j 
    \big\|^2.
  \end{align*}
  At this point we notice that
  \begin{align*}
   &2\eta h \sum_{j=1}^{n-1}\triple \Delta g^{j}\triple^2
    + 2\eta h \sum_{j=0}^{n-2}\triple \Delta g^{j}\triple^2
    + 2\eta h \triple \Delta g^{n-1} \triple^2 + h \triple g^{0}
    \triple^2 - 4 \eta h \sum_{j=2}^n\triple \Delta g^j \triple^2\\
    &\quad
    \leq
    - 4 \eta h \triple \Delta g^n \triple^2
    + (2\eta + 1) h \triple \Delta g^0 \triple^2 
    + 4\eta h\triple \Delta g^1 \triple^2.
  \end{align*}
  In addition,
  since $1 - h (4L+1) > 1 - h_B (4L+1) = \frac{1}{2}$ and $h\le h_B<\tfrac12$ and
  $\eta>\tfrac12$ as well as $\| 2 E^1 + E^0 \|^2 \le 5 ( \| E^1\|^2  + \| E^0
  \|^2 )$  we obtain 
  after some elementary transformations the inequality
  \begin{align*}
    &\|E^n\|^2
      + 
      \frac{2\eta - 1}{\eta} \|2E^n-E^{n-1}\|^2
      +
      h \triple \Delta g^n \triple^2
    \leq
    \max\Big\{ 30, 4 \eta + 2 , \frac{16\eta}{2\eta-1} \Big\}\\
    &\quad\times\Big(\|E^0\|^2+\|E^{1}\|^2 
    + h \triple \Delta g^{0} \triple^2 
    +  h\triple \Delta g^1 \triple^2   + \rho_h^\mathrm{BDF2}(V)\Big)\\
    &\qquad + 4 \max\Big\{  (1+2L)  ,  \frac{\eta}{2\eta-1} \Big\} h
    \sum_{j=2}^{n-1}\Big(\big\| E^j \big\|^2 
    + \frac{2\eta - 1}{\eta} \big\| 2 E^{j} - E^{j-1} \big\|^2 \Big).
  \end{align*}
  The proof is completed by an application of \eqref{eq:Gronwall}. \qed
\end{proof}

\subsection{Consistency of the BDF2 scheme}
\label{subsec:consistency}

In this subsection we bound the local truncation error of the exact solution. 

\begin{theorem}\label{thm:consistency}
Let Assumption~\ref{as:fg} hold and let $X$ be the solution to \eqref{eq:SDE}.
Then there exists $C>0$ such that
\begin{align*}
  \rho_h^\mathrm{BDF2}(X|_{\tau_h})
  \leq
  C
  h
  ,\quad
  h\in(0,1). 
\end{align*}
\end{theorem}

\begin{proof}
In this proof we write $\rho_i^j:=\rho_{h,i}^j(X|_{\tau_h})$, $i\in\{1,2,3\}$,
$j\in\{2,\dots,N_h\}$, $h\in(0,1)$. From the definition of $\rho_h$ we see that
it suffices to show that 
\begin{equation}
  \label{eq:conv1}
  \begin{split}
    \max_{j\in\{2,\dots,N_h\}} 
    & \Big(
    \big\|
      \rho_1^j
    \big\|^2
    +
    \big\|
      \rho_2^j+\rho_3^j
    \big\|^2
    +
    \frac{1}{h} \big\|P_h^{j-1}\rho_2^j\big\|^2
    +
    \frac{1}{h} \big\|P_h^{j-2}\rho_3^j\big\|^2 
  \Big)
  \leq C h^2.
\end{split}
\end{equation}
It holds for $j\in\{2,\dots,N_h\}$ that
\begin{align*}
  \rho_1^j
&=
  \int_{t_{j-1}}^{t_j}
    \big(
      f(X(t_j))
      -
      f(X(s))
    \big)
  \diff{s}
  +
  \int_{t_{j-1}}^{t_j}
    \big(
      g(X(t_{j-1}))
      -
      g(X(s))
    \big)
  \diff{W(s)},\\
 \rho_2^j + \rho_3^j
&=
  \frac12
  \Big(
    \int_{t_{j-1}}^{t_j}
      \big(
        g(X(t_{j-1}))
        -
        g(X(s))
      \big)
    \diff{W(s)}
    -
   \int_{t_{j-1}}^{t_j}
     f(X(s))
   \diff{s}\\
& \qquad
    -
    \int_{t_{j-2}}^{t_{j-1}}
      \big(
        g(X(t_{j-2}))
        -
        g(X(s))
      \big)
    \diff{W(s)}
   +
   \int_{t_{j-2}}^{t_{j-1}}
     f(X(s))
   \diff{s}
  \Big),\\
  P_h^{j-1}
    \rho_2^j
&=
  \frac12
  \E
  \Big[
    \int_{t_{j-1}}^{t_j}
      \big(
        f(X(t_{j-1}))
        -
        f(X(s))
      \big)
    \diff{s}
    \big|
    \mathcal{F}_{t_{j-1}}
  \Big],\\
  P_h^{j-2}
  \rho_3^j
&=
  -
  \frac12
  \E
  \Big[
    \int_{t_{j-2}}^{t_{j-1}}
      \big(
        f(X(t_{j-2}))
        -
        f(X(s))
      \big)
    \diff{s}
    \big|
    \mathcal{F}_{t_{j-2}}
  \Big].
\end{align*}
As in the proof of Theorem~\ref{thm:BEMconsistency}, we note that the estimates
\eqref{eq:bound_f} and \eqref{eq:bound_g} are applicable to \eqref{eq:conv1}
due to the moment bound \eqref{eq:moment}. We further make use of the fact that
$\| \E [ V | \F_{t_{j-1}} ] \| \le \| V \|$ for every $V \in \LOR$ and the bound 
\begin{equation*}
  \begin{split}
    &\Big\| \int_{\tau_1}^{\tau_2} f(X(s)) \diff{s} \Big\| \leq
    C \Big( 1 + \sup_{t\in[0,T]} \big\| X(t) \big\|_{L^{2q}(\Omega;\R^m)}^q
    \Big) |\tau_2-\tau_1|,\quad t_1,t_2\in[0,T],
  \end{split}
\end{equation*}
which is obtained from \eqref{eq:poly_growth}. 
By these estimates and an application of the triangle inequality we directly
deduce \eqref{eq:conv1}. \qed
\end{proof}

\subsection{Mean-square convergence of the BDF2 scheme}
\label{subsec:convergence}

Here we consider the numerical approximations $(X_h^j)_{j=0}^{N_h}$,
$h\in(0,h_B]$, $h_B = \frac{1}{\max\{8L,2\}}$, which are uniquely determined by
the backward difference formula \eqref{eq:BDF2} and a family of initial values
$(X_h^0,X_h^1)_{h\in(0,h_B]}$. 
This family is assumed to satisfy the following assumption.

\begin{assumption}\label{as:initial}
The family of initial values $(X_h^0,X_h^1)_{h\in(0,h_B]}$ satisfies
\begin{align}
\label{eq:initial_cons2}
  X_h^\ell, f(X_h^\ell)
  \in
  L^2(\Omega,\mathcal{F}_{t_\ell},\P;\R^m)
  ,\quad
  g(X_h^\ell)
  \in
  L^2(\Omega,\mathcal{F}_{t_\ell},\P;\R^{m\times d}),
\end{align}
for all $h\in(0,h_B]$, $\ell \in\{0,1\}$, and is consistent of order $\tfrac12$
in the sense that  
\begin{align}\label{eq:initial_cons}
  \sum_{\ell=0}^1\|X(h \ell)-X_h^\ell\|^2
  +
  h\sum_{\ell=0}^1\triple g(X(h \ell)) - g(X_h^\ell) \triple^2
  =\mathcal{O}(h),
\end{align}
as $h\downarrow 0$, where $X$ is the solution to \eqref{eq:SDE}. 
\end{assumption}

We are now ready to state the main result of this section.

\begin{theorem}
  \label{thm:convergence}
  Let Assumptions~\ref{as:fg} and \ref{as:initial} hold, let $X$ be the
  solution to \eqref{eq:SDE} and $(X_h^j)_{j=0}^{N_h}$, $h \in (0,h_B]$, the
  solutions to \eqref{eq:BDF2} with initial values
  $(X_h^0,X_h^1)_{h\in(0,h_B]}$. Under these conditions the BDF2-Maruyama
  method is mean-square convergent of order $\tfrac12$, more precisely, there
  exists $C>0$ such that 
\begin{align*}
  \max_{n\in\{0,\dots,N_h\}}
  \|X_h^n-X(nh)\|
  \leq
  C\sqrt{h}
  ,\quad
  h\in(0,h_B].
\end{align*}
\end{theorem}

\begin{proof}
For $h\in(0,h_B]$, we apply Theorem~\ref{thm:stability} with
$U=(X_h^j)_{j=0}^{N_h}\in\mathcal{G}_h^2$ and
$V=X|_{\tau_h}=(X(t_j))_{j=0}^{N_h}\in\mathcal{G}_h^2$ and get that there is a
constant $C>0$, independent of $h$, such that 
\begin{align*}
  &\|X_h^n-X(nh)\|^2\\
 &\quad
  \le 
  C
  \Bigg(
  \sum_{\ell=0}^1\|X_h^\ell-X(h\ell)\|^2
  +h \sum_{\ell=0}^1
  \triple g(X_h^{\ell}) - g(X(h\ell)) \triple^2
  +\rho_h^\mathrm{BDF2}(X|_{\tau_h}) \Bigg).
\end{align*}
The sums are of order $\mathcal{O}(h)$ by \eqref{eq:initial_cons}. In addition,
the consistency term $\rho_h^\mathrm{BDF2}(X|_{\tau_h})$ is also of order
$\mathcal{O}(h)$ by Theorem~\ref{thm:consistency}. \qed
\end{proof}

\subsection{Admissible initial values for the BDF2-Maruyama scheme}
\label{subsec:BDF2initial}

Assumption~\ref{as:initial} provides an abstract criterion for an admissible
choice of the initial values for the BDF2-Maruyama method such that the
mean-square convergence of order $\frac12$ is ensured.
Here we consider a concrete scheme for the computation of the
second initial value, namely the computation of $X_h^1$ by one step of the
backward Euler-Maruyama method.

\begin{theorem}\label{thm:initial}
  Let Assumption~\ref{as:fg} be fulfilled. Consider a family 
  $(X_h^0)_{h\in(0,h_B]}$ of approximate initial values
  satisfying Assumption~\ref{as:BEMinitial}.
  If $(X_h^1)_{h\in(0,h_B]}$ is determined by one step of the backward
  Euler-Maruyama method, i.e, if for all 
  $h\in(0,h_B]$ the random variable $X_h^1$ solves the equation 
  \begin{align*}
    X_h^1=X_h^0+hf(X_h^1)+g(X_h^0)\Delta_h W^1,
  \end{align*}
  then $(X_h^0,X_h^1)_{h\in(0,h_B]}$ satisfy the conditions of
  Assumption~\ref{as:initial}.
\end{theorem}

\begin{proof}
  The fact that the solution of \eqref{eq:BEM} belongs to $\mathcal{G}_h^2$
  proves \eqref{eq:initial_cons2} of Assumption~\ref{as:initial}. By
  Theorem~\ref{thm:BEMstability} it holds that 
  \begin{align*}
    & \sum_{\ell=0}^1\|X(t_\ell)-X_h^\ell\|^2
    + h \sum_{\ell=0}^1 \triple g(X(t_\ell)) - g(X_h^\ell) \triple^2\\
    &\quad \leq \big( 1+ Ce^{2(1+2L)h} \big) 
    \Big(
      \|X(0)-X_h^0\|^2
      +
      h \triple g(X(0)) - g(X_h^0) \triple^2
      +
      \rho_h^\mathrm{BEM}(X|_{\tau_h})
    \Big).
  \end{align*}
  From Theorem~\ref{thm:BEMconsistency} and Assumption~\ref{as:BEMinitial} the
  right hand side is of order $\mathcal{O}(h)$ as $h\downarrow0$, and this
  proves \eqref{eq:initial_cons}. \qed
\end{proof}

\begin{remark}
  Consider the same assumption as in Theorem~\ref{thm:initial}. From the
  H\"older continuity of the solution $X$ of \eqref{eq:SDE} and
  Assumption~\ref{as:BEMinitial} it holds that  
  \begin{align*}
    \|X(h)-X_h^0\|
    \leq
    \|X(h)-X(0)\| + \|X(0)-X_h^0\|
    \leq 
    C\sqrt{h}.
  \end{align*}
  Therefore, also the choice $X_h^1:=X_h^0$ satisfies the conditions of
  Assumption~\ref{as:initial} and, therefore, is feasible in
  terms of the asymptotic rate of convergence. However, numerical simulations
  similar to those in Section~\ref{sec:numerics} indicate that, although the
  experimental convergence rates behave as expected, this simple choice of the
  second initial value leads to a significantly larger error compared to
  $X_h^1$ being generated by one step of the backward Euler-Maruyama method. 
\end{remark}

\section{Numerical experiments}
\label{sec:numerics}

In this section we perform some numerical experiments which illustrate the
theoretical results from the previous sections. In
Subsection~\ref{subsec:3over2} we consider the $\tfrac32$-volatility model from
finance, which is a one dimensional equation. In Subsection~\ref{subsec:2D} we
do computations for a two dimensional dynamics which mimics the form and
properties of the discretization of a stochastic partial differential equation,
like the Allen-Cahn equation. 

\subsection{An example in one dimension: the $\tfrac32$-volatility model}
\label{subsec:3over2}

Hereby we consider the
stochastic differential equation 
\begin{align}
  \label{eq:3_2}
  \begin{split}
    \diff{X}(t) &= \big[ X(t) - \lambda X(t) | X(t) | \big] \diff{t} + \sigma
    |X(t)|^{\frac{3}{2}} \diff{W(t)}, \quad t \in [0,T],\\
    X(0) &= X_0, 
  \end{split}
\end{align}
with $m = d = 1$, $\lambda > 0$, $\sigma \in \R$, and $X_0 \in \R$. For
positive initial conditions this equation is also known as the
$\frac{3}{2}$-volatility model \cite{goard2013,henryL2007}. From the quadratic
growth of the drift it holds that $q=2$ in Assumption~\ref{as:fg} and, as the
reader can check, the coercivity condition~\eqref{eq:bound_fg} is valid for
$L=1$ provided that $\lambda\geq\tfrac{4q-3}{2} \sigma^2 =\tfrac52 \sigma^2$.
>From the calculation in \cite[Appendix]{sabanis2013b} it holds for all
$x_1,x_2\in\R$ that 
\begin{align*}
  &\big( f(x_1) - f(x_2) , x_1 - x_2 \big)
  +
  \eta\big| g(x_1) - g(x_2) \big|^2\\
  &\quad
  \leq
  |x_1-x_2|^2+(2\sigma^2\eta-\lambda)
  \big( |x_1| + |x_2| \big)\big(|x_1|-|x_2|\big)^2.
\end{align*}
The global monotonicity condition \eqref{eq:onesided} is therefore satisfied
with $L=1$ and $\eta\leq\tfrac{\lambda}{2\sigma^2}$. As we require
$\eta>\tfrac12$ this imposes the condition $\lambda>\sigma^2$ and altogether we
have that Assumption~\ref{as:fg} is valid for $L=1$, $q=2$, 
$\eta\in(\tfrac12,\tfrac{\lambda}{2\sigma^2})$ provided that
$
\lambda \geq \tfrac52 \sigma^2
$.

In our experiments we approximate the strong error of convergence
for the explicit Euler-Maruyama method (EulM) (see \cite{kloeden1999}), the
backward Euler-Maruyama method (BEM), and the BDF2-Maruyama method (BDF2),
respectively. More precisely, we approximate the root mean square error by a
Monte Carlo simulation based on $M = 10^6$ samples, that is
\begin{align*}
  \mathrm{error}(h) := \max_{0 \le n \le N_h} \Big( \frac{1}{M} \sum_{m = 1}^M
  \big| X^{(m)}(h n) - X_h^{n,(m)} \big|^2 \Big)^{\frac{1}{2}} \approx
  \max_{0 \le n \le N_h} \big\| X(n h) - X_h^n \big\|,  
\end{align*}
where for every $m \in \{1,\ldots,M\}$ the processes $X^{(m)}$ and
$(X_h^{n,(m)})_{n = 0}^{N_h}$ denote independently generated copies of $X$ and
$X_h$, respectively. Here we set $h := \frac{T}{N_h}$ and for the number of
steps $N_h$ we use the values $\{ 25 \cdot 2^{k}\; : \; k =0,\ldots,7\}$,
i.e., $N_h$ ranges from $25$ to $3200$. Since there is no explicit
expression of the exact solution to \eqref{eq:3_2} available, we replace
$X^{(m)}$ in the error computation by a numerical reference solution generated
by the BDF2-Maruyama method with $N_{\mathrm{ref}} = 25 \cdot
2^{12}$ steps.  

As already discussed in Section~\ref{sec:wellposedness}, in every time step of
the implicit schemes we have to solve a nonlinear equation of the form
\begin{align}
  \label{eq:quadeq}
  X_h^{j} - h \beta f(X_h^j) &= R_h^j, 
\end{align}
for $X_h^j$. Here, we have
$f(x) = x - \lambda x|x|$ and $g(x) = \sigma
|x|^{\frac{3}{2}}$ for $x \in \R$ and 
\begin{align}\label{eq:Rh_BEM}
  \beta = 1,& \quad R_h^j = X^{j-1}_h + g(X^{j-1}) \Delta_h W^j, 
\end{align}
for the backward Euler-Maruyama method, and
\begin{align}\label{eq:Rh_BDF2}
  \beta = \frac{2}{3}, & \quad R_h^j = \frac{4}{3} X_h^{j-1} - \frac{1}{3}
  X_h^{j-2} + g(X_h^{j-1}) \Delta_h W^j - \frac{1}{3} g(X_h^{j-2}) \Delta_h
  W^{j-1},
\end{align}
for the BDF2-Maruyama method.
For the $\frac{3}{2}$-volatility model it turns out that \eqref{eq:quadeq} is a
simple quadratic equation, which can be solved explicitly. 
Depending on the sign of the right hand side $X_h^j$ is given by 
\begin{align*}
  X_h^j =
  \begin{cases}
    - \frac{1 - \beta h}{2 \beta h \lambda}
    + \Big( \big(\frac{1 - \beta h}{2 \beta h \lambda}\big)^2 
    + \frac{R_h^j}{ \beta h \lambda} \Big)^{\frac{1}{2}},
    & \text{if } R_h^j \ge 0,\\
    \frac{1 - \beta h}{2 \beta h \lambda} 
    - \Big( \big(\frac{1 - \beta h}{2 \beta h \lambda}\big)^2 
    -\frac{R_h^j}{\beta h \lambda} \Big)^{\frac{1}{2}},
    & \text{if } R_h^j < 0.
  \end{cases}
\end{align*}
As the first initial value we set $X_h^0 \equiv X_0$ for both schemes. In
addition, we generate the second initial value for the BDF2-Maruyama method
by one step of the BEM method as proposed in Section~\ref{subsec:BDF2initial}.
Note that the computation of one step of the BDF2-Maruyama method is, up to
some additional operations needed for the evaluation of $R_h^j$, as costly as 
one step of the BEM method.  

\small
\begin{table}[t]
  \caption{Non-stiff case without noise: $\lambda = 4$, $\sigma = 0$.}
  \label{tab:32vol_exp3}
  \begin{tabular}{p{1.1cm}p{1.5cm}p{1.2cm}p{1.5cm}p{1.2cm}p{1.5cm}p{0.9cm}}
     & EulM &      & BEM &      & BDF2 &  \\ 
     \hline\noalign{\smallskip}
     $N_h$     & error & EOC     & error & EOC     & error & EOC \\ 
     \noalign{\smallskip}\hline\noalign{\smallskip}
      25      & 0.024635 &      & 0.020186 &      & 0.010594 & \\ 
      50  & 0.011619 & 1.08  & 0.010528 & 0.94  & 0.003739 & 1.50 \\ 
      100  & 0.005659 & 1.04  & 0.005388 & 0.97  & 0.001134 & 1.72 \\ 
      200  & 0.002793 & 1.02  & 0.002726 & 0.98  & 0.000325 & 1.80 \\ 
      400  & 0.001388 & 1.01  & 0.001371 & 0.99  & 0.000088 & 1.89 \\ 
      800  & 0.000692 & 1.00  & 0.000688 & 1.00  & 0.000023 & 1.93 \\ 
      1600  & 0.000345 & 1.00  & 0.000344 & 1.00  & 0.000006 & 1.96 \\ 
      3200  & 0.000173 & 1.00  & 0.000172 & 1.00  & 0.000002 & 1.98 \\ 
      \noalign{\smallskip}\hline
  \end{tabular}
\end{table}
\normalsize

\small
\begin{table}[t]
  \caption{Non-stiff case with smaller noise intensity: $\lambda = 4$, $\sigma =
  \frac{1}{3}$.} 
  \label{tab:32vol_exp1}
  \begin{tabular}{p{1.1cm}p{1.5cm}p{1.2cm}p{1.5cm}p{1.2cm}p{1.5cm}p{0.9cm}}
     & EulM &      & BEM &      & BDF2 &  \\ 
     \hline\noalign{\smallskip}
     $N_h$     & error & EOC     & error & EOC     & error & EOC \\ 
     \noalign{\smallskip}\hline\noalign{\smallskip}
      25      & 0.026853 &      & 0.020812 &      & 0.011949 & \\ 
      50  & 0.012788 & 1.07  & 0.011020 & 0.92  & 0.004961 & 1.27 \\ 
      100  & 0.006398 & 1.00  & 0.005816 & 0.92  & 0.002662 & 0.90 \\ 
      200  & 0.003334 & 0.94  & 0.003115 & 0.90  & 0.001695 & 0.65 \\ 
      400  & 0.001818 & 0.87  & 0.001733 & 0.85  & 0.001140 & 0.57 \\ 
      800  & 0.001051 & 0.79  & 0.001016 & 0.77  & 0.000785 & 0.54 \\ 
      1600  & 0.000647 & 0.70  & 0.000632 & 0.69  & 0.000546 & 0.52 \\ 
      3200  & 0.000417 & 0.63  & 0.000411 & 0.62  & 0.000380 & 0.52 \\ 
      \noalign{\smallskip}\hline
  \end{tabular}
\end{table}
\normalsize

\small
\begin{table}[t]
  \caption{Non-stiff case with higher noise intensity: $\lambda = 4$, $\sigma =
  1$.}
  \label{tab:32vol_exp2}
  \begin{tabular}{p{1.1cm}p{1.5cm}p{1.2cm}p{1.5cm}p{1.2cm}p{1.5cm}p{0.9cm}}
     & EulM &      & BEM &      & BDF2 &  \\ 
     \hline\noalign{\smallskip}
     $N_h$     & error & EOC     & error & EOC     & error & EOC \\ 
     \noalign{\smallskip}\hline\noalign{\smallskip}
      25      & 0.069083 &      & 0.048076 &      & 0.048450 & \\ 
      50  & 0.039639 & 0.80  & 0.032517 & 0.56  & 0.032775 & 0.56 \\ 
      100  & 0.024776 & 0.68  & 0.022244 & 0.55  & 0.022382 & 0.55 \\ 
      200  & 0.016454 & 0.59  & 0.015563 & 0.52  & 0.015617 & 0.52 \\ 
      400  & 0.011187 & 0.56  & 0.010873 & 0.52  & 0.010894 & 0.52 \\ 
      800  & 0.007770 & 0.53  & 0.007656 & 0.51  & 0.007665 & 0.51 \\ 
      1600  & 0.005412 & 0.52  & 0.005375 & 0.51  & 0.005377 & 0.51 \\ 
      3200  & 0.003790 & 0.51  & 0.003778 & 0.51  & 0.003778 & 0.51 \\ 
      \noalign{\smallskip}\hline
  \end{tabular}
\end{table}
\normalsize

In all simulations the initial condition of equation \eqref{eq:3_2} is set to
be $X_0 = 1$, while the length of the time interval equals $T= 1$.
Regarding the choice of the noise intensity $\sigma$ let us note that in the
deterministic situation with $\sigma = 0$ it is well-known that the BDF2 method
converges with order $2$ to the exact solution. In the stochastic case with
$\sigma > 0$, however, the order of convergence reduces asymptotically to
$\frac{1}{2}$ due to the presence of the noise. Hence, the BDF2-Maruyama
method offers apparently no advantage over the backward Euler-Maruyama method.
But, as it has already been observed in \cite{buckwar2006}, one still benefits
from the higher deterministic order of convergence if the intensity of the
noise is small compared to the step size of the numerical scheme. To illustrate
this effect we use three different noise levels in our simulations: the
deterministic case $\sigma = 0$, a small noise intensity with $\sigma =
\frac{1}{3}$, and a higher intensity with $\sigma = 1$.

Moreover, the equation \eqref{eq:3_2} behaves stiffer in
the sense of numerical analysis if the value for
$\lambda$ is increased. Since implicit numerical schemes like the backward
Euler method and the BDF2 method are known to behave more stable in this
situation than explicit schemes, we will perform our simulations with two
different values for $\lambda$: The non-stiff case with $\lambda = 4$ and the
stiff case with $\lambda = 25$. Note that the condition $
\lambda \geq \tfrac52 \sigma^2$ is satisfied for all combinations of $\lambda$
and $\sigma$.

Further, to better illustrate the effect of the parameter $\lambda$ on
explicit schemes explains why we also included the explicit Euler-Maruyama
method in our simulations. Although this scheme is actually known to be
divergent for SDEs involving superlinearly growing drift- and diffusion
coefficient functions, see \cite{hutzenthaler2011}, it nonetheless often yields
reliable numerical results. But let us stress that the observed  
experimental convergence of the explicit Euler-Maruyama scheme is purely
empirical and does not indicate its convergence in the sense of
\eqref{eq:msconv}.

The first set of numerical results are displayed in Tables~\ref{tab:32vol_exp3}
to \ref{tab:32vol_exp2}, which are concerned with the non-stiff case $\lambda =
4$. In Table~\ref{tab:32vol_exp3} we see the errors computed  in the
deterministic case $\sigma=0$. As expected the explicit Euler scheme and the
backward Euler method perform equally well, while the experimental errors of
the BDF2 method are much smaller. This is also indicated by the
\emph{experimental order of convergence} (EOC) which is defined for successive
step sizes and errors by 
\begin{align*}
  \mathrm{EOC} = \frac{\log( \mathrm{error}(h_i) ) -
  \log(\mathrm{error}(h_{i-1}))}{ \log(h_i) - \log(h_{i-1})}.
\end{align*}
As expected the numerical results are in line with the theoretical
orders.

In Table~\ref{tab:32vol_exp1} the noise intensity is increased to $\sigma =
\frac{1}{3}$. Here we see that for larger step sizes, that is $N_h \in \{25,
50\}$, the erros are only slightly larger than in the deterministic case. In
fact, for the two Euler methods the discretization error of the drift part
seems to dominate the total error for almost all step sizes as the errors
mostly coincide with those in Table~\ref{tab:32vol_exp3}. On the other hand, the
BDF2-Maruyama method performs significantly better for larger and medium sized
step sizes. Only on the two finest refinement levels $N_h \in \{1600, 3200\}$
the estimated errors of all three schemes are of the same magnitude.
This picture changes drastically in Table~\ref{tab:32vol_exp2}, which shows the
result of the same experiment but with $\sigma = 1$. Here the errors of all
schemes agree for almost all step sizes and the BDF2-Maruyama method is no
longer superior.

\small
\begin{table}[t]
  \caption{Stiff case without noise: $\lambda = 25$, $\sigma = 0$.} 
  \label{tab:32vol_exp6}
  \begin{tabular}{p{1.1cm}p{1.5cm}p{1.2cm}p{1.5cm}p{1.2cm}p{1.5cm}p{0.9cm}}
     & EulM &      & BEM &      & BDF2 &  \\ 
     \hline\noalign{\smallskip}
     $N_h$     & error & EOC     & error & EOC     & error & EOC \\ 
     \noalign{\smallskip}\hline\noalign{\smallskip}
      25      & 0.475184 &      & 0.114050 &      & 0.114050 & \\ 
      50  & 0.157860 & 1.59  & 0.067366 & 0.76  & 0.062722 & 0.86 \\ 
      100  & 0.054660 & 1.53  & 0.038126 & 0.82  & 0.027090 & 1.21 \\ 
      200  & 0.024244 & 1.17  & 0.020389 & 0.90  & 0.010049 & 1.43 \\ 
      400  & 0.011541 & 1.07  & 0.010594 & 0.94  & 0.003426 & 1.55 \\ 
      800  & 0.005640 & 1.03  & 0.005404 & 0.97  & 0.001017 & 1.75 \\ 
      1600  & 0.002789 & 1.02  & 0.002730 & 0.98  & 0.000289 & 1.81 \\ 
      3200  & 0.001387 & 1.01  & 0.001372 & 0.99  & 0.000078 & 1.90 \\ 
      \noalign{\smallskip}\hline
  \end{tabular}
\end{table}
\normalsize

\small
\begin{table}[t]
  \caption{Stiff case with smaller noise intensity: $\lambda = 25$, $\sigma =
  \frac{1}{3}$.}
  \label{tab:32vol_exp5}
  \begin{tabular}{p{1.1cm}p{1.5cm}p{1.2cm}p{1.5cm}p{1.2cm}p{1.5cm}p{0.9cm}}
     & EulM &      & BEM &      & BDF2 &  \\ 
     \hline\noalign{\smallskip}
     $N_h$     & error & EOC     & error & EOC     & error & EOC \\ 
     \noalign{\smallskip}\hline\noalign{\smallskip}
      25   & 0.477006 &       & 0.114345 &       & 0.114345 & \\ 
      50   & 0.159402 & 1.58  & 0.067506 & 0.76  & 0.062802 & 0.86 \\ 
      100  & 0.055190 & 1.53  & 0.038191 & 0.82  & 0.027099 & 1.21 \\ 
      200  & 0.024422 & 1.18  & 0.020420 & 0.90  & 0.010129 & 1.42 \\ 
      400  & 0.011622 & 1.07  & 0.010614 & 0.94  & 0.003471 & 1.55 \\ 
      800  & 0.005679 & 1.03  & 0.005416 & 0.97  & 0.001074 & 1.69 \\ 
      1600 & 0.002811 & 1.01  & 0.002739 & 0.98  & 0.000351 & 1.61 \\ 
      3200 & 0.001401 & 1.00  & 0.001379 & 0.99  & 0.000182 & 0.95 \\ 
      \noalign{\smallskip}\hline
  \end{tabular}
\end{table}
\normalsize

\small
\begin{table}
  \caption{Stiff case with higher noise intensity: $\lambda = 25$, $\sigma =
  1$.}
  \label{tab:32vol_exp4}
  \begin{tabular}{p{1.1cm}p{1.5cm}p{1.2cm}p{1.5cm}p{1.2cm}p{1.5cm}p{0.9cm}}
     & EulM &      & BEM &      & BDF2 &  \\ 
     \hline\noalign{\smallskip}
     $N_h$     & error & EOC     & error & EOC     & error & EOC \\ 
     \noalign{\smallskip}\hline\noalign{\smallskip}
      25      & 0.492313 &      & 0.117432 &      & 0.117432 & \\ 
      50  & 0.171504 & 1.52  & 0.069266 & 0.76  & 0.064308 & 0.87 \\ 
      100  & 0.060257 & 1.51  & 0.039290 & 0.82  & 0.028057 & 1.20 \\ 
      200  & 0.026558 & 1.18  & 0.021182 & 0.89  & 0.011774 & 1.25 \\ 
      400  & 0.012821 & 1.05  & 0.011208 & 0.92  & 0.005207 & 1.18 \\ 
      800  & 0.006477 & 0.98  & 0.005938 & 0.92  & 0.002991 & 0.80 \\ 
      1600  & 0.003414 & 0.92  & 0.003215 & 0.89  & 0.001922 & 0.64 \\ 
      3200  & 0.001897 & 0.85  & 0.001814 & 0.83  & 0.001295 & 0.57 \\ 
      \noalign{\smallskip}\hline
  \end{tabular}
\end{table}
\normalsize

The second set of experiments shown in Tables~\ref{tab:32vol_exp6} to
\ref{tab:32vol_exp4} are concerned with the stiff case $\lambda = 25$ while all
other parameters remain unchanged. At first glance we see that the explicit
Euler-Maruyama method performs much worse than the two implicit methods for
$N_h \in \{25, 50, 100, 200\}$ in the deterministic case $\sigma = 0$. 
This even stays true when noise is present. On the other hand, the BDF2
method clearly performs best in the deterministic case although the
experimental order of convergence increases rather slowly to $2$ compared to
the non-stiff case in Table~\ref{tab:32vol_exp3}. Further note that the error
of the BEM method and the BDF2 method agree for $N_h = 25$. This is explained
by the fact that the second initial value of the multi-step method is generated
by the BEM method and, apparently, this is where the error is largest for both
schemes. 

Moreover, we observe that the errors in Table~\ref{tab:32vol_exp5} with $\sigma
= \frac{1}{3}$ are of the same magnitude as those in Table~\ref{tab:32vol_exp6}.
Due to the larger value of $\lambda$ the presence of the noise only
seems to have a visible impact on the error of the BDF2-Maruyama method with
$N_h = 3200$. Hence, the BEM method performs significantly worse than the
BDF2-Maruyama scheme for all larger values of $N_h$.

In contrast to the non-stiff case this behaviour does not change so drastically
when the noise intensity is increased to $\sigma = 1$. In
Table~\ref{tab:32vol_exp4} we still observe a better performance of the
BDF2-Maruyama method, although the estimated errors are seemingly affected by
the presence of a stronger noise.

To sum up, in our numerical experiments the BDF2-Maruyama method  and the two
Euler methods performed equally well if the equation is non-stiff and driven 
by a higher noise intensity. In all other tested scenarios (with stiffness
and/or with small noise intensity) the BDF2-Maruyama method is often superior
to the two Euler methods in terms of the experimental error. Hence,
our observations confirm the results reported earlier in \cite{buckwar2006}.  


\subsection{An example in two dimensions: a toy discretization of an SPDE}
\label{subsec:2D}

Here we consider the two dimensional equation
\begin{align*}
&\diff X_{1}(t)
  +
  \frac12
  \big(
    (1+\lambda)X_1(t)
    +
    (1-\lambda)X_2(t)
  \big)
  \diff t
  =
  \big(
    X_1(t)-X_1(t)^3
  \big)
  \diff t
  +
    \sigma
    X_1(t)^2
  \diff{W_1(t)},\\
&\diff X_2(t)
  +
  \frac12
  \big(
    (1-\lambda)X_1(t)
    +
    (1+\lambda)X_2(t)
  \big)
  \diff t
  =
  \big(
    X_2(t)-X_2(t)^3
  \big)
  \diff t
  +
    \sigma
    X_2(t)^2
  \diff{W_2(t)},
\end{align*}
for $t\in(0,T]$, $X(0)=X_0\in\R^2$ with $\sigma\geq0$ and $\lambda>>0$. We write this equation in the form
\begin{align}\label{eq:toySPDE}
  \diff X(t)
  +AX(t) \diff{t}
  =f(X(t))\diff t
  +
  g(X(t))\diff W(t),
  \quad
  t\in(0,T],
  \quad
  X(0)=X_0,
\end{align}
with $A$ being the positive and symmetric $2\times 2$-matrix
\begin{align*}
  A=
  \frac1{\sqrt{2}}
  \left[
    \begin{array}{cc}
      1 & 1\\
      1 & -1
    \end{array}
  \right]
  \left[
    \begin{array}{cc}
      1 & 0\\
      0 & \lambda
    \end{array}
  \right]
  \frac1{\sqrt{2}}
  \left[
    \begin{array}{cc}
      1 & 1\\
      1 & -1
    \end{array}
  \right]
  =
  \frac12
  \left[
    \begin{array}{ll}
      1+\lambda & 1-\lambda\\
      1-\lambda & 1+\lambda
    \end{array}
  \right],
\end{align*}
and the non-linearities $f\colon \R^2\to \R^2$, $g\colon\R^2\to\R^{2\times 2}$ given by
\begin{align*}
  f(x)=
  \left[
    \begin{array}{c}
      x_1-x_1^3\\
      x_2-x_2^2
    \end{array}
  \right],
  \quad
  g(x)=
  \sigma
  \left[
    \begin{array}{cc}
      x_1^2 & 0\\
      0& x_2^2
    \end{array}
  \right],
  \quad
  X_0=(x_1,x_2)\in\R^2.
\end{align*}

The reason why we are interested in \eqref{eq:toySPDE} is its similarity to the
equation obtained when discretizing a stochastic partial differential equation
with a Galerkin method. The matrix $-A$ is a substitute for a discrete
Laplacian, which is a symmetric and negative definite matrix. With a Galerkin
approximation in $k$ dimensions the eigenvalues
$\lambda_1<\lambda_2,\dots,\lambda_{k-1}<\lambda_k$ satisfy that $\lambda_k>>0$
is very large, causing a stiff system. Our matrix $A$ is chosen to mimic this
stiffness. Moreover, our choice of $f$ is due to its similarity with the
nonlinearity of the Allen-Cahn equation. We take the diffusion coefficient to
be quadratic in order to demonstrate an example with superlinear growth.

Assumption~\ref{as:fg} is valid for all $\lambda\geq 0$ and
$\sigma\in[0,\sqrt{2}/3)$ with $L=1$, $q=3$ and $\eta=\tfrac1{2\sigma^2}$. The
coercivity condition \eqref{eq:bound_fg} is left to the reader to check and it
is in fact it that determines the upper bound for $\sigma$. To verify the
global monotonicity condition \eqref{eq:onesided}, we first notice that for
$x,y\in \R^2$, it holds that
\begin{align*}
&\big(
    f(x)-f(y),x-y
  \big)\\
& \quad
  =
  |x-y|^2
  -(x_1^3-y_1^3)
  (x_1-y_1)
  -(x_2^3-y_2^3)
  (x_2-y_2)\\
& \quad
   =
  |x-y|^2
  -(x_1^2 + x_1 y_1 + y_1^2)
  (x_1-y_1)^2
  -(x_2^2 + x_2 y_2+ y_2^2)
  (x_2-y_2)^2,
\end{align*}
and that
\begin{align*}
  |g(x)-g(y)|_{\HS}^2
&=
  \Tr((g(x)-g(y))^*(g(x)-g(y)))\\
&=
  \sigma^2
  \Tr
  \left(
    \left[
      \begin{array}{cc}
        (x_1^2-y_1^2)^2 & \#\\
        \# & (x_2^2-y_2^2)^2  
      \end{array}
    \right]
  \right)\\
&=
   \sigma^2
   \big(
     (x_1^2-y_1^2)^2 +(x_2^2-y_2^2)^2
   \big)\\
&=
   \sigma^2
   \big(
   (x_1-y_1)^2(x_1+y_1)^2 
   +
   (x_2-y_2)^2(x_2+y_2)^2
   \big)\\
&=
   \sigma^2
   \big(
   (x_1-y_1)^2(x_1^2+2x_1y_1+y_1^2) 
   + 
   (x_2-y_2)^2(x_2^2+2x_2y_2+y_2^2)
   \big).
\end{align*}
Since it also holds that 
$
  -
  (
    A(x-y),x-y
  )
  \leq0
$
we get for all $x,y\in\R^2$ that
\begin{align*}
&\big(
    A(x-y)+f(x)-f(y)
    ,x-y
  \big)
  +
  \eta
  |g(x)-g(y)|_{\HS}^2\\
&\quad
  \leq
  |x-y|^2
  +
  (\eta\sigma^2-1)
  \big(
    (x_1^2+y_1^2)(x_1-y_1)^2
    +
    (x_2^2+y_2^2)(x_2-y_2)^2
  \big)\\
&\qquad
  +
  (2\eta\sigma^2-1)
  \big(
    x_1y_1(x_1-y_1)^2 + x_2 y_2 (x_2-y_2)^2
  \big)
\end{align*}
Under the assumption that $2\eta\sigma^2=1$ we obtain that
\begin{align*}
&\big(
    A(x-y)+f(x)-f(y)
    ,x-y
  \big)
  +
  \eta
  |g(x)-g(y)|_{\HS}^2
  \leq
  |x-y|^2,
  \quad x,y\in\R^2,
\end{align*}
which proves the global monotonicity for $L=1$.

\begin{figure}\label{fig:sol_path}
  \includegraphics[width=\linewidth]{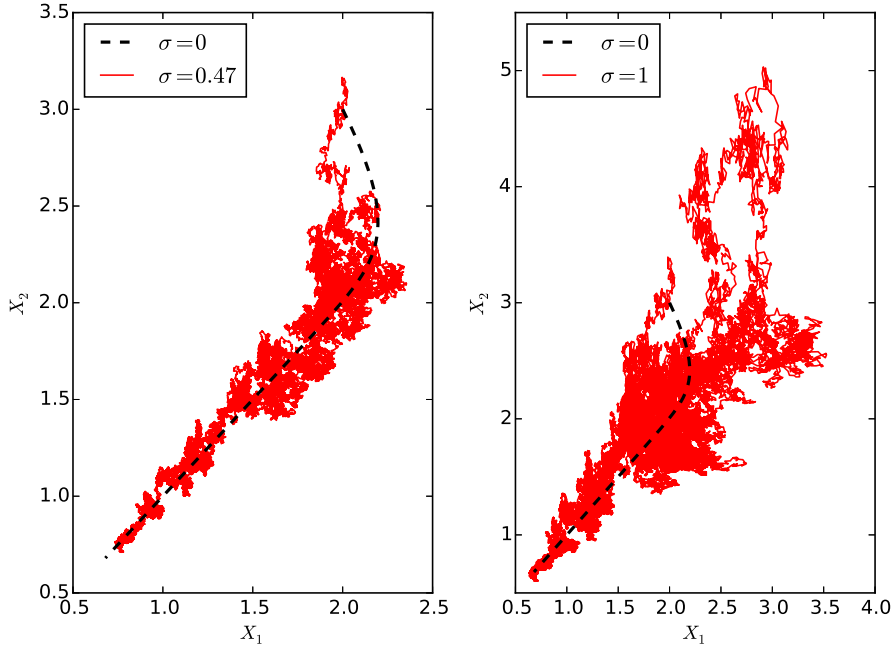}
  \caption{Sample trajectories computed with BDF2 and $N=25\cdot2^{12}$ time
  steps for different noise intensities. The same sample path of the noise is
  used in both plots.}  
\end{figure}

\begin{figure}\label{fig:proj}
  \includegraphics[width=\linewidth]{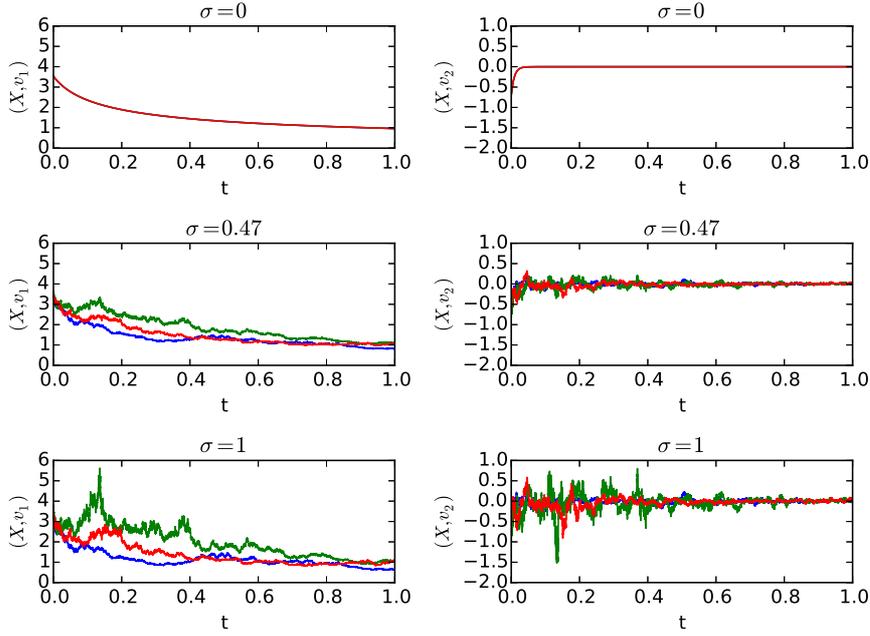}
  \caption{Projections of several sample paths with
  different noise intensity onto the eigenvectors $v_1$
  and $v_2$ of the matrix $A$, respectively.} 
\end{figure}

We use the experimental setup of Subsection~\ref{subsec:3over2}. As the
equation \eqref{eq:quadeq} is not explicitly solvable with $f$ and $g$ from the
present subsection we use $K=5$ Newton iterations in each time step to obtain
an approximate solution. More precisely, these iterations $\tilde
X_h^{j,0},\dots,\tilde X_h^{j,K}=X_h^j$ are given $\tilde X_h^{j,0} =
X_h^{j-1}$ and for $k\in\{1,\dots,K\}$ by 
$\tilde X_h^{j,k}=\tilde X_h^{j,k-1}-(D\Phi_h^j(X_h^{j,k-1}))^{-1}\Phi_h^j(X_h^{j,k-1})$, where
$\Phi_h^j(x)=x-\beta h (f(x)-Ax)-R_h^j$ and $D\Phi_h^j$ is the Jacobian. With $c=\tfrac{1-\lambda}2$
it holds that
\begin{align*}
  (D\Phi_h^j(x))^{-1} \Phi_h^j(x)
&=
  \frac1
  {
    (1-\beta h (c-3x_1^2))
    (1-\beta h (c- 3x_2^2))
    -
    (\beta h c)^2
  }\\
&\quad
  \times
  \left[
    \begin{array}{l}
      (1-\beta h(c-3x_2^2))(x_1-\beta h (c(x_1-x_2)-x_1^3)-R_{h,1}^j)\\
      \quad-\beta h c(x_2-\beta h (c(x_2-x_1)-x_2^3)-R_{h,2}^j)\\
      (1 - \beta h (c-3x_1^2))(x_2-\beta h (c(x_2-x_1)-x_2^3)-R_{h,2}^j)\\
      \quad-\beta h c(x_1-\beta h (c(x_1-x_2)-x_1^3)-R_{h,1}^j)
    \end{array}
  \right].
\end{align*}
This is used for the Newton iterations.

We perform three experiments, all with $\lambda=96$, $T=1$, $X_0=[2,3]^t$, one
without noise, i.e., $\sigma=0$, one with small noise intensity
$\sigma=0.47$, which is just below the threshold $\sqrt{2}/3$ for
$\sigma$, allowed for the theoretical results to be valid, and one with large
noise intensity $\sigma=1$ to see how the methods compare outside the allowed
parameter regime. Comparing Table~\ref{tab:toySPDE1} and \ref{tab:toySPDE2} we
observe that the errors differ very little. This suggests that the noise is
negligible for the small noise case and therefore does not effect the dynamics
much, but this suggestion is false. This is clear from
Figure~\ref{fig:sol_path}, where a typical path is shown and in
Figure~\ref{fig:proj}, where the same solution path, together with two other
solution paths are projected in the directions of the eigenvectors
$v_1=\tfrac1{\sqrt2}[1,1]^t$ and $v_2=\tfrac1{\sqrt2}[1,-1]^t$, corresponding
to the eigenvalues $1$ and $\lambda$, respectively, of the matrix $A$.
Direction $v_2$ is the stiff direction. This expresses itself in
Figure~\ref{fig:proj} by a strong drift towards zero of $(X,v_2)$, while
$(X, v_1)$ is more sensitive to the noise. 

For $N_h=25$ the CFL-condition $|1-\lambda h|<1$ is not satisfied while for
$N_h\geq50$ it is. Table~\ref{tab:toySPDE1} shows, in the case of no noise,
explosion of the Euler-Maruyama method, for the crude refinement levels for
which the CFL condition is not satisfied. The backward Euler-Maruyama and BDF2
methods work for all refinement levels and perform better than the
Euler-Maruyama method, but only the BDF2 method performs significantly better.
For small noise Table~\ref{tab:toySPDE2} shows essentially the same errors as
for those without noise, only with slightly worse performance for the
Euler-Maruyama scheme.  
Taking into account that the computational effort for the BEM
and BDF2-Maruyama schemes are essentially the same our results
show the latter to be superior for this problem. Our results confirm the
conclusion from the previous subsection that the BDF2-Maruyama scheme performs
better for stiff equations with small noise. For $\sigma=1$ the Euler-Maruyama
scheme explodes for all refinement-levels and BDF2 has lost most of its
advantage over BEM. For the crudest step size the error is even higher than for
BEM.

\small
\begin{table}[t]
  \caption{Without noise: $\sigma = 0$.}
  \label{tab:toySPDE1}
  \begin{tabular}{p{1.1cm}p{1.5cm}p{1.2cm}p{1.5cm}p{1.2cm}p{1.5cm}p{0.9cm}}
     & EulM &      & BEM &      & BDF2 &  \\ 
     \noalign{\smallskip}\hline\noalign{\smallskip}
     $N_h$     & error & EOC     & error & EOC     & error & EOC \\ 
     \noalign{\smallskip}\hline\noalign{\smallskip}
      25      & - &      & 0.095559 &      & 0.062955 & \\ 
      50  & 0.512850 & -  & 0.052993 & 0.85  & 0.031538 & 1.00 \\ 
      100  & 0.036046 & 3.83  & 0.028182 & 0.91  & 0.013142 & 1.26 \\ 
      200  & 0.016245 & 1.15  & 0.014630 & 0.95  & 0.005002 & 1.39 \\ 
      400  & 0.007861 & 1.05  & 0.007472 & 0.97  & 0.001784 & 1.49 \\ 
      800  & 0.003875 & 1.02  & 0.003779 & 0.98  & 0.000575 & 1.63 \\ 
      1600  & 0.001925 & 1.01  & 0.001901 & 0.99  & 0.000169 & 1.76 \\ 
      3200  & 0.000959 & 1.00  & 0.000953 & 1.00  & 0.000048 & 1.83 \\ 
  \end{tabular}
\end{table}
\normalsize

\small
\begin{table}[t]
  \caption{Small noise intensity: $\sigma = 0.47$.} 
  \label{tab:toySPDE2}
  \begin{tabular}{p{1.1cm}p{1.5cm}p{1.2cm}p{1.5cm}p{1.2cm}p{1.5cm}p{0.9cm}}
     & EulM &      & BEM &      & BDF2 &  \\ 
     \hline\noalign{\smallskip}
     $N_h$     & error & EOC     & error & EOC     & error & EOC \\ 
     \noalign{\smallskip}\hline\noalign{\smallskip}
      25      & - &      & 0.091651 &      & 0.056638 & \\ 
      50  & - & -  & 0.051324 & 0.84  & 0.029411 & 0.94 \\ 
      100  & 0.044452 & -  & 0.027599 & 0.90  & 0.012246 & 1.26 \\ 
      200  & 0.019432  & 1.19  & 0.014429 & 0.94  & 0.004621 & 1.41 \\ 
      400  & 0.009300 & 1.06  & 0.007403 & 0.96  & 0.001613 & 1.52 \\ 
      800  & 0.004538 & 1.04  & 0.003792 & 0.97  & 0.000523 & 1.62 \\ 
      1600  & 0.002253 & 1.01  & 0.001910 & 0.99  & 0.000187 & 1.48 \\ 
      3200  & 0.001124 & 1.00  & 0.000961 & 0.99  & 0.000095 & 0.97 \\ 
      \noalign{\smallskip}\hline
  \end{tabular}
\end{table}
\normalsize

\small
\begin{table}[t]
  \caption{Large noise intensity: $\sigma = 1$.} 
  \label{tab:toySPDE3}
  \begin{tabular}{p{1.1cm}p{1.5cm}p{1.2cm}p{1.5cm}p{1.2cm}p{1.5cm}p{0.9cm}}
     & EulM &      & BEM &      & BDF2 &  \\ 
     \hline\noalign{\smallskip}
     $N_h$     & error & EOC     & error & EOC     & error & EOC \\ 
     \noalign{\smallskip}\hline\noalign{\smallskip}
      25      & - &      & 0.085359 &      & 0.107572 & \\ 
      50  & - & -  & 0.054110 & 0.66  & 0.054215 & 0.99 \\ 
      100  & - & -  & 0.030864 & 0.81  & 0.027929 & 0.96 \\ 
      200  & - & -  & 0.016582 & 0.89  & 0.014925 & 0.90 \\ 
      400  & - & -  & 0.008928 & 0.89  & 0.007853 & 0.93 \\ 
      800  & - & -  & 0.004635 & 0.95  & 0.004010 & 0.97 \\ 
      1600  & - & -  & 0.002372 & 0.97  & 0.002135 & 0.91 \\ 
      3200  & - & -  & 0.001221 & 0.96  & 0.001127 & 0.92 \\ 
      \noalign{\smallskip}\hline
  \end{tabular}
\end{table}
\normalsize

\begin{acknowledgements}
The authors wish to thank Etienne Emmrich for bringing the
identity \eqref{eq:emmrich2} to our attention and for informing us about its
connection to the BDF2-scheme. Stig Larsson and Chalmers University of
Technology are acknowledged for our use of the computational resource Ozzy.
This research was carried out in the framework of {\sc Matheon} supported by
Einstein Foundation Berlin. 
\end{acknowledgements}



\end{document}